\documentclass[11pt]{article}

\usepackage[english]{babel}
\usepackage[letterpaper,top=2.2cm,bottom=2.2cm,left=3cm,right=3cm]{geometry}

\usepackage{amsmath,amssymb,amsthm,mathtools}
\usepackage{bm}
\usepackage{graphicx}
\usepackage{subcaption}
\usepackage{placeins}
\usepackage{enumitem}
\usepackage{algorithm}
\usepackage{algpseudocode}
\usepackage[colorlinks=true,allcolors=blue]{hyperref}

\theoremstyle{plain}
\newtheorem{theorem}{Theorem}[section]
\newtheorem{lemma}[theorem]{Lemma}

\newtheorem{corollary}[theorem]{Corollary}
\theoremstyle{definition}
\newtheorem{definition}[theorem]{Definition}
\theoremstyle{remark}
\newtheorem{remark}[theorem]{Remark}

\newcommand{\RR}{\mathbb{R}}

\newcommand{\M}{\mathfrak{M}}

\newcommand{\LC}{\mathrm{LC}}

\title{Compressed Sensing with Quantized Tensor Trains (QTTs)}
\author{%
  Jingchun Shao\\[0.35em]
  {\small Department of Mathematics, Texas A\&M University}\\
  {\small College Station, TX 77843, USA}\\
  {\small Email: \href{mailto:jingchun.shao@tamu.edu}{\nolinkurl{jingchun.shao@tamu.edu}}}%
}
\date{}

\begin{document}
\maketitle

\begin{abstract}
The storage and recovery of a vector of length \(2^d\) can become
prohibitively expensive as \(d\) grows, a manifestation of the curse of
dimensionality. For certain structured functions and multiscale quantities,
the resulting discretized vectors admit compact quantized tensor train (QTT)
representations after binary tensorization: a length-\(2^d\) vector is reshaped
into a \(d\)-way binary tensor with low tensor-train (TT)
rank. Motivated by this
construction, we study compressed sensing of tensorized signals that combine
bounded TT rank with sparse low-order interaction structure, using randomly
sampled Walsh--Hadamard measurements. For this structured model class, we
establish a uniform
restricted isometry property and derive a measurement condition guaranteeing
uniqueness of noiseless recovery. We also propose a structure-aware
initializer for the alternating linear scheme (ALS), obtained by projecting
the adjoint backprojection onto the low-order interaction space before
applying TT-SVD. We prove a uniform initialization error bound whose
measurement requirement, for fixed structural parameters, grows polynomially
with the tensor order \(d\) rather than with the ambient dimension \(2^d\).
Numerical experiments support the predicted RIP and initialization scaling
and demonstrate more reliable ALS recovery.

\end{abstract}


\section{Introduction}

Recovering a high-dimensional signal from a limited number of measurements is
a central problem in inverse problems, signal processing, and scientific
computing. Without additional structural assumptions, unique recovery of a
general vector in $\RR^N$ requires at least $N$ linear measurements. Classical
compressed sensing reduces this burden by assuming coordinate sparsity. If
$\mathbf{x}\in\RR^N$ has at most $q$ nonzero entries, one observes
\[
    \mathbf{y}
    =
    \mathbf{A}\mathbf{x}
    +
    \boldsymbol{\eta},
\]
where $\mathbf{A}\in\RR^{m\times N}$, $m\ll N$, and
$\boldsymbol{\eta}\in\RR^m$ represents measurement noise or outliers. When
$\mathbf{A}$ satisfies a
restricted isometry property (RIP), sparse vectors can be recovered by
$\ell_1$ minimization or greedy
methods~\cite{candes2005decoding,candes2006robust,foucart2013mathematical}.
Dense Gaussian and Bernoulli matrices can achieve near-optimal measurement
rates, but storing and applying them becomes costly in large ambient
dimensions. Subsampled Fourier matrices are an important structured
alternative: they support fast matrix--vector multiplication through the fast
Fourier transform and arise naturally in frequency-domain sensing, notably
magnetic resonance imaging (MRI)~\cite{lustig2007sparse}. In the binary
setting considered here, the Walsh--Hadamard transform provides a natural
alternative: its recursive Kronecker structure likewise permits a butterfly
implementation with \(\mathcal{O}(N\log N)\) complexity and aligns naturally
with binary tensorization.

Let $N=2^d$ and let $\mathbf{x}^\star\in\RR^N$ be the unknown vector. Its
binary tensorization
\[
    \mathcal{X}^\star
    =
    \operatorname{ten}(\mathbf{x}^\star)
    \in\RR^{2\times\cdots\times2}
\]
is defined precisely in Section~\ref{subsec:notation}. A tensor-train (TT)
representation of this tensorized vector is called a quantized tensor train
(QTT).\footnote{Here, ``quantized'' refers to binary tensorization and is
distinct from scalar quantization of measurements; see Hou and
Ng~\cite{hou2025quantized} for tensor recovery from quantized modewise
measurements.} We assume that
$\mathcal{X}^\star$ has bounded TT rank and belongs to a sparse low-order
interaction model, as formalized in
Section~\ref{subsec:structured-recovery-model}. Each binary mode represents one
bit of the original vector index. For function samples on a grid with $2^d$
points, these bits can be associated with scales ranging from coarse to fine.
Low TT rank then models limited dependence across these
levels~\cite{shinaoka2023multiscale}. If the ranks are bounded by $r$, the QTT
storage is $\mathcal{O}(dr^2)$ rather than $\mathcal{O}(2^d)$; see
Section~\ref{subsec:tt-decomposition}.

Complementing the low-TT-rank prior, the sparse low-order interaction model
assumes that $\mathcal{X}^\star$ has a sparse expansion in components involving
only a few binary modes. This assumption is motivated by physical models
dominated by a small number of low-order
interactions~\cite{rabitz1999general,drautz2019atomic}. For a generic
tensorized vector, it represents structured sparsity among the binary index
variables. A closely related combination of sparse low-order interaction
structure and TT compressibility is used for density estimation by
Peng et al.~\cite{peng2024tensor}, while a related low-order cluster basis and
sparse-plus-low-rank compression are used for operator learning by Khoo et
al.~\cite{khoo2025fastoperator}. Together, the two assumptions combine
global QTT compressibility with sparse low-order interaction structure.

We acquire $m$ randomly selected Walsh--Hadamard measurements. Let
\[
    \mathbf{H}_2
    \coloneqq
    \frac{1}{\sqrt{2}}
    \begin{bmatrix}
        1 & 1\\
        1 & -1
    \end{bmatrix},
    \qquad
    \mathbf{H}
    \coloneqq
    \mathbf{H}_2^{\otimes d}
    \in\RR^{N\times N},
\]
and let $\mathbf{S}\in\RR^{m\times N}$ select coordinates independently and
uniformly with replacement. With the normalization used throughout the
analysis, the noiseless measurements are
\begin{equation}
    \mathbf{y}
    =
    \mathbf{A}\mathbf{x}^\star,
    \qquad
    \mathbf{A}
    \coloneqq
    \sqrt{\frac{N}{m}}\,
    \mathbf{S}\mathbf{H}.
\end{equation}
Because $\mathbf{H}_2$ is the normalized two-point discrete Fourier
transform, this measurement model can equivalently be viewed as random
sampling of the $d$-dimensional discrete Fourier transform of
$\mathcal{X}^\star$ on a $2\times\cdots\times2$ grid. Given $\mathbf{y}$ and
the sampled rows, our goal is to recover $\mathbf{x}^\star$, or equivalently
its tensorization $\mathcal{X}^\star$.

The present work is closely related to two lines of research: low-rank tensor recovery and tensor regression.

\paragraph{Tensor recovery.}
Through binary tensorization, the recovery problem above can be viewed as
low-rank tensor recovery from general linear measurements. Rauhut, Schneider,
and Stojanac~\cite{rauhut2016lowranktensor} introduced iterative
hard-thresholding algorithms for recovering low-rank tensors in several
formats and analyzed these algorithms under a tensor restricted isometry
property (TRIP). Their bounds also
cover partial Fourier maps combined with random sign flips. For tensor
recovery and other optimization problems in the TT format, Holtz, Rohwedder,
and Schneider~\cite{holtz2012alternating} developed the alternating linear
scheme (ALS), which is the recovery algorithm used in this work. An important
special case of low-rank tensor recovery is tensor completion, where the
measurement operator observes a subset of tensor entries.
Steinlechner~\cite{steinlechner2016riemannian} formulated fixed-TT-rank tensor
completion as Riemannian optimization, and Cai, Li, and
Xia~\cite{cai2022provable} subsequently developed a provable second-order
spectral initialization and established linear convergence under a nearly
optimal sampling condition.

More recently, for general linear measurements of low-TT-rank tensors, Qin,
Wakin, and Zhu~\cite{qin2024guaranteed} established local linear convergence
of factorized Riemannian gradient descent under TT-RIP, using the spectral
initializer
$\operatorname{TT\text{-}SVD}_{\mathbf{r}}(\mathcal{A}^*\mathbf{y})$.
Qin and Zhu~\cite{qin2025robust} later treated arbitrary outliers using an
$\ell_1$ loss and Riemannian subgradient methods. Our initializer uses the same
adjoint backprojection followed by TT-SVD, but inserts a projection onto the
order-$s$ interaction space between them, giving
$\operatorname{TT\text{-}SVD}_{\mathbf{r}}(
\mathcal{P}_{\mathcal{V}_s}(\mathcal{A}^*\mathbf{y}))$.
In our experiments, the added interaction-space projection yields lower
initialization error and more
reliable ALS recovery than the unprojected spectral initializer; it also
underlies the uniform initialization guarantee developed below.

\paragraph{Tensor regression.}
Using the tensor notation introduced in Section~2.1, let
$f\coloneqq\mathcal{H}(\mathcal{X}^\star)$. As explained in
Section~\ref{subsec:structured-recovery-model}, $f$ admits a sparse expansion
in the Walsh tensor-product basis, truncated to interactions of order at most
$s$. Moreover,
\[
    \mathbf{y}
    =
    \sqrt{\frac{N}{m}}\,
    \mathbf{S}\operatorname{vec}(f),
\]
so the observations are random point samples of $f$. This interpretation
connects our problem to sparse low-rank approximation of multivariate
functions from point evaluations. Chevreuil et
al.~\cite{chevreuil2015least} construct sparse canonical low-rank
tensor-product approximations from random function evaluations using
regularized least squares and greedy rank-one enrichment. In related work of analysis of variance (ANOVA) decompositions, sparse approximation from scattered samples has been studied using Fourier-based least-squares and regularization methods~\cite{potts2021approximation,bartel2022grouped}.

The coefficient tensor of a tensor-product expansion can itself be
represented and compressed in tensor-network formats; see,
e.g.,~\cite{trunschke2025weighted}. Particularly relevant is the block-sparse
TT regression framework of G{\"o}tte, Schneider, and
Trunschke~\cite{gotte2021blocksparse}. They represent the coefficients of a
multivariate polynomial tensor-product expansion in TT format and fit the
model from point evaluations using an ALS-type least-squares method. Thus,
both settings infer a TT-structured coefficient tensor from pointwise
observations of the corresponding tensor-product function. The structural
assumptions are different, however: G{\"o}tte et al.\ use a polynomial
tensor-product basis with block sparsity induced by polynomial degree,
whereas we use the Walsh/ANOVA cluster basis and further assume sparsity
among the active low-order interaction components. Their analysis focuses on
the resulting block-sparse TT representation and its complexity, while we
establish uniform RIP and noiseless uniqueness guarantees for subsampled
Walsh--Hadamard measurements and develop a structure-aware initializer for
ALS.

A closely related contemporaneous work is presented in Chapter~6 of Peng's
dissertation~\cite[Chapter~6]{peng2026fast}. Both works represent sparse
low-order interactions in an ANOVA tensor-product basis and combine this
structure with a TT representation, so their structural starting points are
closely aligned. At the level of the functional representation and
point-sampling observation model, our setting is the binary two-basis
specialization of Peng's framework. The methods and theoretical guarantees,
however, are different. Peng combines kernel ridge regression, TT
compression, coefficient-based feature sampling, and a reduced
sparse-regression refit. In contrast, we study direct recovery under a
randomly subsampled Walsh--Hadamard operator, establish a uniform RIP and
noiseless uniqueness guarantee for the joint structured model, and develop a
structure-aware initializer for fixed-rank ALS together with a uniform
initialization error bound.

\subsection{Contributions}
The contributions of this work are summarized as follows.
\begin{itemize}
    \item This work provides a compressed-sensing framework for recovering
    vectors whose binary tensorizations combine bounded TT rank with sparse
    low-order interaction structure.
    \item A uniform RIP is established for the randomly sampled
    Walsh--Hadamard operator on the resulting model class and is later used to
    prove uniqueness of noiseless recovery.
    \item A structure-aware initializer for ALS is proposed and shown to
    satisfy a uniform initialization error bound; for fixed structural
    parameters, the required number of measurements grows polynomially in $d$.
    \item Numerical experiments examine the predicted RIP and initialization
    scalings and compare the structure-aware initializer with alternative
    initializations in end-to-end ALS recovery.

\end{itemize}

\subsection{Organization}
The remainder of this paper is organized as follows. Section~2 introduces the
notation, preliminaries, and problem setup. Section~3 presents the ALS recovery
method and the proposed initialization. Section~4 establishes the structured
RIP and uniqueness of noiseless recovery. Section~5 analyzes the
structure-aware initialization. Section~6 reports the numerical experiments,
and Section~7 concludes the paper.

\section{Preliminaries}
\subsection{Notation}
\label{subsec:notation}
To distinguish the tensor formulation from its vectorized
representation, we use the typographical conventions summarized below.

\begin{center}
\begin{tabular}{lll}
\hline
Notation type & Objects & Examples \\
\hline
Lowercase italic & Scalars
    & $d,m,N,r$ \\
Bold lowercase & Vectors
    & $\mathbf{x},\mathbf{y}$ \\
Bold uppercase & Matrices
    & $\mathbf{A},\mathbf{H}$ \\
Calligraphic uppercase & Tensors and tensor operators
    & $\mathcal{X},\mathcal{A},\mathcal{H}$ \\
Fraktur uppercase & Model classes
    & $\mathfrak{M}_{r,s,k}(C)$ \\
\hline
\end{tabular}
\end{center}

We use $\mathbb{R}$ and $\mathbb{C}$ to denote the real and complex
number fields, respectively. For a positive integer $n$, we write
$[n]\coloneqq\{1,\ldots,n\}$. We use $\otimes$ for the Kronecker product
and~$\circ$ for the tensor outer product of vectors; their relation under
our vectorization convention is specified below.

\paragraph{Tensors and vectorization.}
The tensor order is denoted by $d$. Since every mode has size two, the
ambient vector dimension is $N\coloneqq 2^d$. For
$\mathcal{X}\in\mathbb{R}^{2\times\cdots\times2}$, we write
\begin{equation}
    \mathbf{x}
    \coloneqq
    \operatorname{vec}(\mathcal{X})
    \in
    \mathbb{R}^N
\end{equation}
for its vectorization. More precisely, for
$i_1,\ldots,i_d\in\{1,2\}$, we define
\begin{equation}
    \bigl[\operatorname{vec}(\mathcal{X})\bigr]_j
    =
    \mathcal{X}(i_1,\ldots,i_d),
    \qquad
    j
    =
    1+\sum_{k=1}^d(i_k-1)2^{k-1}.
\end{equation}
Thus, the first tensor index varies fastest. We denote the inverse
tensorization map by
\begin{equation}
    \operatorname{ten}
    \coloneqq
    \operatorname{vec}^{-1},
    \qquad
    \operatorname{ten}\!\left(
        \operatorname{vec}(\mathcal{X})
    \right)
    =
    \mathcal{X}.
\end{equation}
For matrices $\mathbf{A}=[a_{ij}]\in\mathbb{R}^{m\times n}$ and
$\mathbf{B}\in\mathbb{R}^{p\times q}$, their Kronecker product is the
block matrix
\begin{equation}
    \mathbf{A}\otimes\mathbf{B}
    \coloneqq
    [a_{ij}\mathbf{B}]_{i=1,\ldots,m;\,j=1,\ldots,n}
    \in
    \mathbb{R}^{mp\times nq}.
\end{equation}
For vectors $\mathbf{a}^{(k)}\in\mathbb{R}^{n_k}$,
$k=1,\ldots,d$, their tensor outer product is the $d$-way rank-one
tensor
\begin{equation}
    \mathcal{Z}
    \coloneqq
    \mathbf{a}^{(1)}
    \circ\cdots\circ
    \mathbf{a}^{(d)}
    \in
    \mathbb{R}^{n_1\times\cdots\times n_d},
    \qquad
    \mathcal{Z}(i_1,\ldots,i_d)
    =
    \prod_{k=1}^d a^{(k)}_{i_k}.
\end{equation}
Under the vectorization convention above,
\begin{equation}
    \operatorname{vec}(\mathcal{Z})
    =
    \mathbf{a}^{(d)}
    \otimes\cdots\otimes
    \mathbf{a}^{(1)}.
\end{equation}
Thus, $\circ$ constructs a multiway tensor, whereas $\otimes$ denotes
the corresponding product in vectorized coordinates. The reversed order
in the last display follows because the first tensor index varies
fastest.

For vectors $\mathbf{u},\mathbf{v}\in\mathbb{R}^n$, define the Euclidean
inner product by
$\langle\mathbf{u},\mathbf{v}\rangle_2
\coloneqq\mathbf{u}^{\mathsf T}\mathbf{v}$.
For
$\mathcal{X},\mathcal{Y}\in\mathbb{R}^{2\times\cdots\times2}$, the
Frobenius inner product is defined as
\begin{equation}
    \langle\mathcal{X},\mathcal{Y}\rangle_F
    \coloneqq
    \sum_{i_1=1}^2\cdots\sum_{i_d=1}^2
    \mathcal{X}(i_1,\ldots,i_d)
    \mathcal{Y}(i_1,\ldots,i_d).
\end{equation}
The corresponding Frobenius norm satisfies
\begin{equation}
    \|\mathcal{X}\|_F
    =
    \sqrt{\langle\mathcal{X},\mathcal{X}\rangle_F}
    =
    \|\operatorname{vec}(\mathcal{X})\|_2.
\end{equation}
For vectors and matrices, $\|\cdot\|_2$ denotes the Euclidean norm and
the spectral norm, respectively.

Using these inner products, the adjoint
$\mathcal{L}^*:\mathbb{R}^m\to\mathbb{R}^{2\times\cdots\times2}$ of a
linear operator
$\mathcal{L}:\mathbb{R}^{2\times\cdots\times2}\to\mathbb{R}^m$ is
defined by
\begin{equation}
    \langle\mathcal{L}(\mathcal{X}),\mathbf{y}\rangle_2
    =
    \langle\mathcal{X},\mathcal{L}^*(\mathbf{y})\rangle_F
\end{equation}
for all $\mathcal{X}$ and $\mathbf{y}$. For a real matrix $\mathbf{L}$,
the adjoint is its transpose: $\mathbf{L}^*=\mathbf{L}^{\mathsf T}$.

\paragraph{Walsh--Hadamard transform.}
Recall the normalized Walsh--Hadamard matrix $\mathbf{H}$ introduced above.
We denote its corresponding modewise tensor transform by $\mathcal{H}$, where
\begin{equation}
    \operatorname{vec}\!\left(\mathcal{H}(\mathcal{X})\right)
    =
    \mathbf{H}\operatorname{vec}(\mathcal{X}).
\end{equation}
Every entry of $\mathbf{H}$ has magnitude $2^{-d/2}=N^{-1/2}$.
Moreover, $\mathbf{H}$ is real, symmetric, and orthogonal, so
\begin{equation}
    \mathbf{H}^{-1}
    =
    \mathbf{H}^*
    =
    \mathbf{H},
    \qquad
    \mathcal{H}^{-1}
    =
    \mathcal{H}^*
    =
    \mathcal{H}.
\end{equation}

\subsection{Tensor-train decomposition and TT-SVD}
\label{subsec:tt-decomposition}

We first recall the TT representation and then describe its sequential
construction by TT-SVD~\cite{Oseledets2011}. The discussion allows general
mode sizes, although this paper focuses on $n_1=\cdots=n_d=2$.

\paragraph{Tensor-train representation.}
Let $\mathcal{X}\in\mathbb{R}^{n_1\times\cdots\times n_d}$. A TT
representation consists of cores
$\mathcal{G}_k\in\mathbb{R}^{r_{k-1}\times n_k\times r_k}$,
$k=1,\ldots,d$, with boundary ranks $r_0=r_d=1$. For $i_k=1,\ldots,n_k$,
write $\mathbf{G}_k(i_k)\coloneqq\mathcal{G}_k(:,i_k,:)
\in\mathbb{R}^{r_{k-1}\times r_k}$. Then
\[
\begin{aligned}
    \mathcal{X}(i_1,\ldots,i_d)
    &=
    \mathbf{G}_1(i_1)\mathbf{G}_2(i_2)\cdots\mathbf{G}_d(i_d) \\
    &=
    \sum_{\alpha_1=1}^{r_1}\cdots\sum_{\alpha_{d-1}=1}^{r_{d-1}}
    \mathcal{G}_1(1,i_1,\alpha_1)
    \mathcal{G}_2(\alpha_1,i_2,\alpha_2)
    \cdots
    \mathcal{G}_d(\alpha_{d-1},i_d,1).
\end{aligned}
\]
The matrix product is a scalar because $r_0=r_d=1$; we write the represented
tensor compactly as $\mathcal{X}=[\mathcal{G}_1,\ldots,\mathcal{G}_d]$.
In quantum many-body physics, the TT format is commonly referred to as a
matrix product state (MPS) representation~\cite{schollwock2011density}. When
the tensor is obtained by tensorizing a vector, its TT representation is
called a QTT representation, also known as a quantics tensor
train~\cite{oseledets2010qtt}.

For $k=1,\ldots,d-1$, let $\mathbf{X}^{\langle k\rangle}$ be the unfolding
that groups the first $k$ tensor indices into the row index and the remaining
indices into the column index. Its dimensions and the TT ranks are
\[
\begin{aligned}
    \mathbf{X}^{\langle k\rangle}
    &\in
    \mathbb{R}^{(n_1\cdots n_k)\times(n_{k+1}\cdots n_d)},\\
    r_k
    &\coloneqq
    \operatorname{rank}\!\left(\mathbf{X}^{\langle k\rangle}\right),
    \qquad
    \operatorname{rank}_{\mathrm{TT}}(\mathcal{X})
    \coloneqq
    (r_1,\ldots,r_{d-1}).
\end{aligned}
\]
These unfolding ranks are the minimal internal dimensions of an exact TT
representation. For a rank vector $\mathbf{r}=(r_1,\ldots,r_{d-1})$, the
inequality $\operatorname{rank}_{\mathrm{TT}}(\mathcal{X})\leq\mathbf{r}$
is understood componentwise; a scalar bound $r$ means that every TT rank is
at most $r$.

The TT representation stores $\sum_{k=1}^d n_k r_{k-1}r_k$ parameters rather
than the $\prod_{k=1}^d n_k$ entries of the full tensor. If $n_k\leq n$ and
$r_k\leq r$, this is $\mathcal{O}(dnr^2)$ storage, or $\mathcal{O}(dr^2)$
in the binary setting.

\paragraph{TT-SVD.}
TT-SVD constructs the cores sequentially. At step $k=1,\ldots,d-1$, it
reshapes the current remainder into a matrix $M_k$, computes
$M_k=U_k\Sigma_kV_k^{\mathsf T}$, reshapes $U_k$ into the $k$th core, and
passes $\Sigma_kV_k^{\mathsf T}$ to the next step; the final remainder is
reshaped into $\mathcal{G}_d$. Without truncation, this produces an exact TT
representation with the unfolding ranks above.

Given a prescribed rank vector $\mathbf{r}=(r_1,\ldots,r_{d-1})$, truncated
TT-SVD truncates the $k$th SVD to rank at most $r_k$. The resulting
approximation, denoted by
$\operatorname{TT\text{-}SVD}_{\mathbf{r}}(\mathcal{X})$, has TT rank at
most $\mathbf{r}$ and satisfies the quasi-optimality estimate
\cite[Corollary~2.4]{Oseledets2011}:
\[
    \left\|
        \mathcal{X}
        -
        \operatorname{TT\text{-}SVD}_{\mathbf{r}}(\mathcal{X})
    \right\|_F
    \leq
    \sqrt{d-1}
    \inf_{\operatorname{rank}_{\mathrm{TT}}(\mathcal{Y})
        \leq\mathbf{r}}
    \|\mathcal{X}-\mathcal{Y}\|_F.
\]
This truncated map is used in the proposed algorithms.

\subsection{Low-order interaction space}
\label{subsec:low-order-interaction-space}
We first introduce a tensor-product basis for
$\mathbb{R}^{2\times\cdots\times2}$. Define the two local basis vectors
\begin{equation}
    \mathbf{e}_0
    \coloneqq
    \begin{bmatrix}1\\0\end{bmatrix},
    \qquad
    \mathbf{e}_1
    \coloneqq
    \begin{bmatrix}0\\1\end{bmatrix}.
\end{equation}
For each $S\subseteq[d]$, define
\begin{equation}
    \mathcal{E}_S
    \coloneqq
    \mathbf{e}^{(S)}_1\circ\cdots\circ\mathbf{e}^{(S)}_d,
    \quad
    \mathbf{e}^{(S)}_j
    \coloneqq
    \begin{cases}
        \mathbf{e}_1, & j\in S,\\
        \mathbf{e}_0, & j\notin S,
    \end{cases}
    \quad j=1,\ldots,d.
\end{equation}
The family $\{\mathcal{E}_S:S\subseteq[d]\}$ is the canonical orthonormal
basis of $\mathbb{R}^{2\times\cdots\times2}$, indexed by subsets
$S\subseteq[d]$ of the binary modes.

The connection with low-order interactions becomes explicit after applying
the Walsh--Hadamard transform. The local factors satisfy
\begin{equation}
    \mathbf{H}_2\mathbf{e}_0
    =
    \frac{1}{\sqrt{2}}
    \begin{bmatrix}1\\1\end{bmatrix},
    \qquad
    \mathbf{H}_2\mathbf{e}_1
    =
    \frac{1}{\sqrt{2}}
    \begin{bmatrix}1\\-1\end{bmatrix}.
\end{equation}
Since $\mathcal{H}$ acts modewise,
\begin{equation}
    \mathcal{H}\mathcal{E}_S
    =
    \left(\mathbf{H}_2\mathbf{e}^{(S)}_1\right)
    \circ\cdots\circ
    \left(\mathbf{H}_2\mathbf{e}^{(S)}_d\right).
\end{equation}
The transformed tensor $\mathcal{H}\mathcal{E}_S$ is the basis function
associated with the $S$-interaction in the ANOVA decomposition under the uniform product measure. Indeed,
$\mathbf{H}_2\mathbf{e}_0$ is constant, whereas
$\mathbf{H}_2\mathbf{e}_1$ has zero average over the two binary values.
Therefore, $\mathcal{H}\mathcal{E}_S$ is constant outside $S$ and has zero
marginal mean in $S$, as required for a pure
$S$-interaction~\cite{takemura1983tensor,Owen2003}.

\begin{definition}[Order-$s$ interaction space]
\label{def:low-order-interaction-space}
For $0\leq s\leq d$, define
\begin{equation}
    \mathcal{V}_s
    \coloneqq
    \left\{
        \sum_{\substack{S\subseteq[d]\\|S|\leq s}}
        c_S\mathcal{E}_S
        \;\middle|\;
        c_S\in\mathbb{R}
    \right\}.
\end{equation}
\end{definition}
Because the cluster basis tensors are orthonormal, the expansion in the
definition is unique. In the tensor domain, $\mathcal{V}_s$ contains only
cluster-basis coefficients indexed by sets of cardinality at most $s$; in the
spatial domain, $\mathcal{H}(\mathcal{V}_s)$ is precisely the order-$s$
truncated ANOVA space. Such low-order interaction
structure arises in many physical and scientific applications, where
high-dimensional functions are governed predominantly by interactions
among small subsets of variables~\cite{rabitz1999general}. This construction
is also closely related to the low-order cluster bases used in the atomic
cluster expansion~\cite{drautz2019atomic} and in recent work on tensor-train
density estimation~\cite{peng2024tensor} and operator
learning~\cite{khoo2025fastoperator}.

For the vectorized formulation, let
\begin{equation}
    V_s
    \coloneqq
    \operatorname{vec}(\mathcal{V}_s)
    \subseteq
    \mathbb{R}^{2^d}.
\end{equation}
We denote the orthogonal projectors onto the tensor and vectorized spaces
by
\begin{equation}
    \mathcal{P}_{\mathcal{V}_s}:
    \mathbb{R}^{2\times\cdots\times2}
    \rightarrow
    \mathcal{V}_s,
    \qquad
    P_{V_s}:
    \mathbb{R}^{2^d}
    \rightarrow
    V_s,
\end{equation}
respectively. They are related through vectorization:
\begin{equation}
    \operatorname{vec}\!\left(
        \mathcal{P}_{\mathcal{V}_s}\mathcal{X}
    \right)
    =
    P_{V_s}\operatorname{vec}(\mathcal{X}).
\end{equation}

\begin{remark}[Dimension of the interaction space]
\label{rem:interaction-space-dimension}
Since the basis tensors are indexed by subsets of cardinality at most
$s$,
\begin{equation}
    D_s
    \coloneqq
    \dim(\mathcal{V}_s)
    =
    \dim(V_s)
    =
    \sum_{\ell=0}^s\binom{d}{\ell}.
\end{equation}
For $1\leq s\leq d$,
\begin{equation}
    D_s
    \leq
    \left(\frac{ed}{s}\right)^s,
\end{equation}
which is the standard combinatorial bound; see, e.g.,~\cite{foucart2022mathematical}.
Thus, for fixed $s$, the dimension grows polynomially in $d$, in
contrast to the full dimension $2^d$.
\end{remark}

\subsection{Structured recovery model}
\label{subsec:structured-recovery-model}
We now formalize the structured recovery problem. The unknown vector and its
binary tensorization are
\begin{equation}
    \mathbf{x}^\star
    \in
    \mathbb{R}^{N},
    \qquad
    \mathcal{X}^\star
    \coloneqq
    \operatorname{ten}(\mathbf{x}^\star)
    \in
    \mathbb{R}^{2\times\cdots\times2}.
\end{equation}
Binary tensorization is an isometric reindexing, so recovering
$\mathbf{x}^\star$ is equivalent to recovering $\mathcal{X}^\star$. We
therefore formulate the structural model and recovery algorithm in tensor form,
where TT rank and sparse low-order interaction structure are explicit, and
pass to vector form only for sampled Hadamard-row calculations. After
recovering $\mathcal{X}^\star$, we obtain the original vector via
$\mathbf{x}^\star=\operatorname{vec}(\mathcal{X}^\star)$.

We impose two assumptions on the tensorization. First,
$\mathcal{X}^\star$ has
bounded TT rank:
\begin{equation}
    \operatorname{rank}_{\mathrm{TT}}(\mathcal{X}^\star)
    \leq r.
\end{equation}

Second, we assume that the cluster-basis expansion can be written as a sum of
at most $k$ components, each involving at most $s$ binary modes. Specifically,
we define
\begin{equation}
    \LC_k(s)
    \coloneqq
    \left\{
        \sum_{t=1}^{k}\mathcal{X}_{S_t}
        \;\middle|\;
        \begin{array}{l}
            S_t\subseteq[d],\quad |S_t|\leq s,\\[2pt]
            \mathcal{X}_{S_t}
            \in
            \operatorname{span}\{\mathcal{E}_I:I\subseteq S_t\},
            \quad t=1,\ldots,k
        \end{array}
    \right\}.
\end{equation}
Each $\mathcal{X}_{S_t}$ is supported on coordinates indexed by
$I\subseteq S_t$, so it has at most $2^{|S_t|}\leq2^s$ nonzero entries. Hence
every tensor in $\LC_k(s)$ has at most $k2^s$ nonzero cluster-basis
coefficients, and $\LC_k(s)$ is a sparse low-order
model contained in $\mathcal{V}_s$. The component decomposition need not be unique because these coordinate
supports may overlap.

Combining the two assumptions, we define
the model class
\begin{equation}
    \M_{r,s,k}(C)
    \coloneqq
    \left\{
        \mathcal{X}
        \in
        \mathbb{R}^{2\times\cdots\times2}:
        \operatorname{rank}_{\mathrm{TT}}(\mathcal{X})\leq r,\
        \mathcal{X}\in\LC_k(s),\
        \|\mathcal{X}\|_F\leq C
    \right\}.
\end{equation}
Here, $r$ bounds the TT ranks; $s$ and $k$ separately control the maximum
interaction order and the number of active interaction components; and
$C$ bounds the signal scale through the Frobenius norm.

We now formalize the measurement model. Draw $\nu_1,\ldots,\nu_m$
independently and uniformly from $[N]$ with replacement, and define
$\mathbf{S}\in\RR^{m\times N}$ by
\begin{equation}
    [\mathbf{S}\mathbf{z}]_\ell
    \coloneqq
    z_{\nu_\ell},
    \qquad
    \mathbf{z}\in\RR^N,\qquad \ell\in[m].
\end{equation}
For $\mathcal{Z}\in\RR^{2\times\cdots\times2}$, set
$\mathcal{S}(\mathcal{Z})\coloneqq
\mathbf{S}\operatorname{vec}(\mathcal{Z})$ and define
\begin{equation}
    \mathbf{A}
    \coloneqq
    \sqrt{\frac{N}{m}}\,\mathbf{S}\mathbf{H},
    \qquad
    \mathcal{A}(\mathcal{Z})
    \coloneqq
    \mathbf{A}\operatorname{vec}(\mathcal{Z})
    =
    \sqrt{\frac{N}{m}}\,
    \mathcal{S}\bigl(\mathcal{H}(\mathcal{Z})\bigr).
\end{equation}
To simplify the subsequent analysis, rescale the Hadamard rows to have
unit-magnitude entries:
\begin{equation}
    \mathbf{h}_j
    \coloneqq
    \sqrt{N}\,\mathbf{H}_{j,:}^{\mathsf T}
    \in\{\pm1\}^N,
    \qquad j\in[N].
\end{equation}
Thus, for $\mathbf{z}=\operatorname{vec}(\mathcal{Z})$,
\begin{equation}
    \label{eq:hadamard-row-measurement}
    [\mathcal{A}(\mathcal{Z})]_\ell
    =
    [\mathbf{A}\mathbf{z}]_\ell
    =
    \frac{1}{\sqrt{m}}
    \left\langle\mathbf{h}_{\nu_\ell},\mathbf{z}\right\rangle_2,
    \qquad \ell\in[m].
\end{equation}
The normalized noiseless model is
\begin{equation}
    \mathbf{y}
    =
    \mathbf{A}\mathbf{x}^\star
    =
    \mathcal{A}(\mathcal{X}^\star).
\end{equation}
The scaling makes $\mathbf{A}$ isotropic, since
$\mathbb{E}[\mathbf{S}^*\mathbf{S}]=(m/N)\mathbf{I}_N$ and hence
$\mathbb{E}[\mathbf{A}^*\mathbf{A}]=\mathbf{I}_N$.
Given $\mathbf{y}$ and the sampled indices $\nu_1,\ldots,\nu_m$, the
objective is to recover $\mathcal{X}^\star\in\M_{r,s,k}(C)$.

\section{Algorithm}

\subsection{Alternating linear scheme}
\label{sec:iter-opt}

We recover the tensorization $\mathcal{X}^\star$ in the TT format using the
alternating linear scheme (ALS). For a prescribed TT-rank vector
$\mathbf{r}=(r_1,\ldots,r_{d-1})$, write the current iterate as
\begin{equation}
    \mathcal{X}
    =
    [\mathcal{G}_1,\ldots,\mathcal{G}_d],
    \qquad
    \mathcal{G}_j
    \in
    \mathbb{R}^{r_{j-1}\times n_j\times r_j},
    \qquad
    r_0=r_d=1.
 \end{equation}
In the binary setting considered here, $n_j=2$ for every $j$. The
fixed-rank least-squares problem is
\begin{equation}
    \min_{\mathcal{G}_1,\ldots,\mathcal{G}_d}
    \frac{1}{2}
    \left\|
        \mathcal{A}
        \bigl([\mathcal{G}_1,\ldots,\mathcal{G}_d]\bigr)
        -
        \mathbf{y}
    \right\|_2^2.
\end{equation}

ALS updates one core while holding all other cores fixed. When updating
the $q$-th core, define the linear insertion map
\begin{equation}
    \mathcal{L}_q(\mathcal{Z})
    \coloneqq
    [\mathcal{G}_1,\ldots,\mathcal{G}_{q-1},
      \mathcal{Z},
      \mathcal{G}_{q+1},\ldots,\mathcal{G}_d]
\end{equation}
and the induced core-to-measurement operator
\begin{equation}
    \mathcal{A}_q(\mathcal{Z})
    \coloneqq
    \mathcal{A}\bigl(\mathcal{L}_q(\mathcal{Z})\bigr).
\end{equation}
Because the TT contraction is linear in any single core, $\mathcal{A}_q$
is linear. The $q$-th core update is therefore
\begin{equation}
\label{eq:iter-opt}
    \min_{
        \mathcal{Z}
        \in
        \mathbb{R}^{r_{q-1}\times n_q\times r_q}
    }
    \frac{1}{2}
    \left\|
        \mathcal{A}_q(\mathcal{Z})
        -
        \mathbf{y}
    \right\|_2^2.
\end{equation}

Let $\mathbf{z}\coloneqq\operatorname{vec}(\mathcal{Z})$. The matrix
representation of $\mathcal{A}_q$ is the matrix
\begin{equation}
    \mathbf{A}_q
    \in
    \mathbb{R}^{m\times(r_{q-1}n_qr_q)}
    \qquad\text{such that}\qquad
    \mathcal{A}_q(\mathcal{Z})
    =
    \mathbf{A}_q\mathbf{z}.
\end{equation}
Consequently, \eqref{eq:iter-opt} reduces to
\begin{equation}
\label{eq:least-squares-form}
    \min_{\mathbf{z}}
    \frac{1}{2}
    \|\mathbf{A}_q\mathbf{z}-\mathbf{y}\|_2^2.
\end{equation}
Its normal equations are
\begin{equation}
    \mathbf{A}_q^*\mathbf{A}_q\mathbf{z}
    =
    \mathbf{A}_q^*\mathbf{y}.
\end{equation}
We solve each least-squares subproblem through its normal equations and reshape
the solution into the updated core. ALS then proceeds to the next core. One
complete update of all $d$ cores constitutes an ALS sweep.

\subsection{Structure-aware initialization}
\label{subsec:rip-initialization}

ALS can be sensitive to its initial TT cores and may converge to a poor
stationary point. We therefore construct a data-informed initialization
using the two structural assumptions of the model.

The construction is motivated by the restricted isometry property (RIP).
For tensors $\mathcal{X}$ in the structured model class,
\begin{equation}
    \|\mathcal{A}(\mathcal{X})\|_2^2
    =
    \left\langle
        \mathcal{A}^*\mathcal{A}(\mathcal{X}),\mathcal{X}
    \right\rangle_F
    \approx
    \|\mathcal{X}\|_F^2.
\end{equation}
Thus, heuristically, the quadratic form of
$\mathcal{A}^*\mathcal{A}$ is close to that of the identity on the model class.
Given $\mathbf{y}=\mathcal{A}(\mathcal{X}^\star)$, this motivates the
backprojection
\begin{equation}
    \mathcal{A}^*\mathbf{y}
    =
    \mathcal{A}^*\mathcal{A}(\mathcal{X}^\star)
\end{equation}
as an initial approximation to $\mathcal{X}^\star$.

Since $\mathcal{A}=\sqrt{N/m}\,\mathcal{S}\mathcal{H}$ and the normalized
Walsh--Hadamard transform is self-adjoint, its adjoint is
\begin{equation}
    \mathcal{A}^*
    =
    \sqrt{\frac{N}{m}}\,\mathcal{H}\mathcal{S}^*.
\end{equation}
The backprojection need not belong to the low-order interaction space, so
we project it onto $\mathcal{V}_s$:
\begin{equation}
    \widetilde{\mathcal{X}}_0
    \coloneqq
    \mathcal{P}_{\mathcal{V}_s}
    \bigl(\mathcal{A}^*\mathbf{y}\bigr)
    \in
    \mathcal{V}_s.
\end{equation}

We then apply truncated TT-SVD both to construct the TT representation
required by ALS and to impose the prescribed TT-rank bound:
\begin{equation}
    \mathcal{X}_{\mathrm{init}}
    =
    [\mathcal{G}^{(0)}_1,\ldots,\mathcal{G}^{(0)}_d]
    \coloneqq
    \operatorname{TT\text{-}SVD}_{\mathbf{r}}
    \bigl(\widetilde{\mathcal{X}}_0\bigr),
    \qquad
    \operatorname{rank}_{\mathrm{TT}}
    (\mathcal{X}_{\mathrm{init}})
    \leq
    \mathbf{r}.
\end{equation}
The resulting cores initialize ALS. The procedure is summarized in
Algorithm~\ref{alg:rip-init}.

\begin{algorithm}[H]
\caption{Structure-aware initialization}
\label{alg:rip-init}
\begin{algorithmic}[1]
\Statex \textbf{Input:} Measurements $\mathbf{y}$, adjoint measurement
operator $\mathcal{A}^*$, interaction order $s$, and target TT-rank
vector $\mathbf{r}$
\State $\widetilde{\mathcal{X}}_0
    \gets
    \mathcal{P}_{\mathcal{V}_s}(\mathcal{A}^*\mathbf{y})$
    \Comment{Low-order projection}
\State $[\mathcal{G}^{(0)}_1,\ldots,\mathcal{G}^{(0)}_d]
    \gets
    \operatorname{TT\text{-}SVD}_{\mathbf{r}}
    (\widetilde{\mathcal{X}}_0)$
    \Comment{TT representation with prescribed ranks}
\Statex \textbf{Output:} Initial TT cores
$\mathcal{G}^{(0)}_1,\ldots,\mathcal{G}^{(0)}_d$
\end{algorithmic}
\end{algorithm}

Since ALS keeps the prescribed ranks fixed, the end-to-end recovery
experiments use the oracle choice
\(\mathbf{r}=\operatorname{rank}_{\mathrm{TT}}(\mathcal{X}^\star)\).
The initialization-scaling experiments instead use the uniform rank bound
\(r_j=k\), \(j=1,\ldots,d-1\).

\section{Restricted isometry and uniqueness}

\subsection{RIP for the structured model class}
\label{sec:rip}

We first prove concentration for a fixed tensor, using the Hadamard-row
representation in \eqref{eq:hadamard-row-measurement}, and then extend the
result uniformly by a covering argument.

\begin{lemma}[Fixed-tensor concentration]
\label{lem:fixed-rip}
Let $\mathcal{X}\in\LC_k(s)\setminus\{0\}$ be fixed. For every
$\varepsilon\in(0,1)$,
\begin{equation}
    \mathbb{P}
    \left(
        \left|
            \|\mathcal{A}(\mathcal{X})\|_2^2
            -\|\mathcal{X}\|_F^2
        \right|
        \geq
        \varepsilon\|\mathcal{X}\|_F^2
    \right)
    \leq
    2\exp\left(-\frac{m\varepsilon^2}{4k2^s}\right).
\end{equation}
\end{lemma}

\begin{proof}
Let $\mathbf{x}=\operatorname{vec}(\mathcal{X})$ and define
\begin{equation}
    Z_\ell
    \coloneqq
    \left|
        \left\langle
            \mathbf{h}_{\nu_\ell},\mathbf{x}
        \right\rangle
    \right|^2,
    \qquad \ell=1,\ldots,m.
\end{equation}
By \eqref{eq:hadamard-row-measurement},
\begin{equation}
    \|\mathcal{A}(\mathcal{X})\|_2^2
    =
    \frac{1}{m}\sum_{\ell=1}^m Z_\ell.
\end{equation}
The variables $Z_1,\ldots,Z_m$ are independent and identically
distributed. Since $\mathbf{H}$ is orthogonal,
\begin{equation}
    \mathbb{E}Z_\ell
    =
    \frac{1}{N}
    \sum_{j=1}^{N}
    \left|
        \left\langle
            \mathbf{h}_j,\mathbf{x}
        \right\rangle
    \right|^2
    =
    \frac{1}{N}
    \|\sqrt{N}\,\mathbf{H}\mathbf{x}\|_2^2
    =
    \|\mathbf{H}\mathbf{x}\|_2^2
    =
    \|\mathbf{x}\|_2^2
    =
    \|\mathcal{X}\|_F^2.
\end{equation}

Because $\mathcal{X}\in\LC_k(s)$, the support of $\mathbf{x}$ is
contained in a set $T\subseteq[N]$ with $|T|\leq k2^s$. Since the
entries of every $\mathbf{h}_j$ have magnitude one,
the Cauchy--Schwarz inequality gives
\begin{equation}
    Z_\ell
    =
    \left|
        \sum_{j\in T}
        (\mathbf{h}_{\nu_\ell})_j x_j
    \right|^2
    \leq
    |T|\,\|\mathbf{x}\|_2^2
    \leq
    k2^s\|\mathcal{X}\|_F^2.
\end{equation}
Set $B\coloneqq k2^s\|\mathcal{X}\|_F^2$. Then
$0\leq Z_\ell\leq B$ and
$|Z_\ell-\mathbb{E}Z_\ell|\leq B$. Moreover,
\begin{equation}
    \operatorname{Var}(Z_\ell)
    \leq
    \mathbb{E}Z_\ell^2
    \leq
    B\,\mathbb{E}Z_\ell
    =
    k2^s\|\mathcal{X}\|_F^4.
\end{equation}
Bernstein's inequality applied to
$Z_\ell-\mathbb{E}Z_\ell$ therefore yields
\begin{align}
    &\mathbb{P}
    \left(
        \left|
            \frac{1}{m}\sum_{\ell=1}^m Z_\ell
            -\|\mathcal{X}\|_F^2
        \right|
        \geq
        \varepsilon\|\mathcal{X}\|_F^2
    \right) \notag\\
    &\qquad\leq
    2\exp\left(
        -\frac{m\varepsilon^2}
        {2(1+\varepsilon/3)k2^s}
    \right)
    \leq
    2\exp\left(-\frac{m\varepsilon^2}{4k2^s}\right),
\end{align}
where the last inequality uses $\varepsilon<1$.
\end{proof}

We next recall the TT covering estimate used in the uniform argument.
It follows by covering a suitably orthogonalized collection of TT
cores; see, for example,~\cite{rauhut2016lowranktensor}.

\begin{lemma}[Covering number for TT tensors]
\label{lem:tt-cover}
Let $d\geq2$, $r\geq1$, and $C>0$, and define
\begin{equation}
    \mathcal{G}_r(C)
    \coloneqq
    \left\{
        \mathcal{X}\in\RR^{2\times\cdots\times2}:
        \operatorname{rank}_{\mathrm{TT}}(\mathcal{X})\leq r,\
        \|\mathcal{X}\|_F\leq C
    \right\}.
\end{equation}
There exists an absolute constant $K>0$ such that, for every
$0<\eta\leq C$,
\begin{equation}
    \mathcal{N}
    \left(
        \mathcal{G}_r(C),
        \|\cdot\|_F,
        \eta
    \right)
    \leq
    \left(
        \frac{Kd\sqrt{r}\,C}{\eta}
    \right)^{2dr^2}.
\end{equation}
\end{lemma}

We now combine the fixed-tensor estimate with this covering bound.

\begin{theorem}[Uniform RIP over the structured tensor class]
\label{thm:uniform-rip}
Let $\varepsilon,\delta\in(0,1)$. If
\begin{equation}
    \label{eq:uniform-measure-bound}
    m
    \geq
    \frac{16k2^s}{\varepsilon^2}
    \left[
        \log\frac{2}{\delta}
        +
        2dr^2
        \log\left(
            \frac{16Kd\sqrt{r}\,k2^s}{\varepsilon}
        \right)
    \right],
\end{equation}
then, with probability at least $1-\delta$,
\begin{equation}
    \left|
        \|\mathcal{A}(\mathcal{X})\|_2^2
        -\|\mathcal{X}\|_F^2
    \right|
    \leq
    \varepsilon\|\mathcal{X}\|_F^2
    \qquad
    \text{for every }\mathcal{X}\in\M_{r,s,k}(C).
\end{equation}
\end{theorem}

\begin{proof}
Because scalar multiplication preserves both TT rank and membership in
$\LC_k(s)$, it is enough to work on the normalized structured set
\begin{equation}
    \mathbb{S}_{r,s,k}
    \coloneqq
    \left\{
        \mathcal{X}\in\RR^{2\times\cdots\times2}:
        \operatorname{rank}_{\mathrm{TT}}(\mathcal{X})\leq r,\
        \mathcal{X}\in\LC_k(s),\
        \|\mathcal{X}\|_F=1
    \right\}.
\end{equation}
Since $\mathbb{S}_{r,s,k}\subseteq\mathcal{G}_r(1)$, there exists an
$\eta$-net $\mathcal{N}_\eta\subseteq\mathbb{S}_{r,s,k}$ satisfying
\begin{equation}
    \label{eq:structured-net-cardinality}
    |\mathcal{N}_\eta|
    \leq
    \mathcal{N}\left(
        \mathcal{G}_r(1),\|\cdot\|_F,\frac{\eta}{2}
    \right)
    \leq
    \left(
        \frac{2Kd\sqrt{r}}{\eta}
    \right)^{2dr^2}.
\end{equation}
Fix $\mathcal{X}\in\mathbb{S}_{r,s,k}$ and choose
$\mathcal{X}^{(p)}\in\mathcal{N}_\eta$ such that
$\|\mathcal{X}-\mathcal{X}^{(p)}\|_F\leq\eta$. Let $\mathbf{x}$ and
$\mathbf{x}^{(p)}$ denote their vectorizations. The triangle inequality
gives
\begin{equation}
    \left|
        \|\mathcal{A}(\mathcal{X})\|_2^2-1
    \right|
    \leq
    \left|
        \|\mathcal{A}(\mathcal{X})\|_2^2
        -
        \|\mathcal{A}(\mathcal{X}^{(p)})\|_2^2
    \right|
    +
    \left|
        \|\mathcal{A}(\mathcal{X}^{(p)})\|_2^2-1
    \right|.
\end{equation}
The second term is the fixed-tensor RIP error at the net point. We
bound the first term deterministically.

\paragraph{Net approximation error.}
By \eqref{eq:hadamard-row-measurement},
\begin{equation}
\begin{aligned}
    &\left|
        \|\mathcal{A}(\mathcal{X})\|_2^2
        -\|\mathcal{A}(\mathcal{X}^{(p)})\|_2^2
    \right|\\
    &\quad\leq
    \frac{1}{m}\sum_{\ell=1}^m
    \left|
        \left\langle
            \mathbf{h}_{\nu_\ell},
            \mathbf{x}-\mathbf{x}^{(p)}
        \right\rangle
    \right|
    \left|
        \left\langle
            \mathbf{h}_{\nu_\ell},
            \mathbf{x}+\mathbf{x}^{(p)}
        \right\rangle
    \right|.
\end{aligned}
\end{equation}
Each of $\mathbf{x}$ and $\mathbf{x}^{(p)}$ has at most $k2^s$
nonzero entries. Their sum and difference are therefore supported on at
most $2k2^s$ coordinates. Moreover,
$\|\mathbf{x}+\mathbf{x}^{(p)}\|_2\leq2$ because both tensors belong to
$\mathbb{S}_{r,s,k}$. Thus, for every sampled row,
\begin{equation}
\begin{aligned}
    \left|
        \left\langle
            \mathbf{h}_{\nu_\ell},
            \mathbf{x}-\mathbf{x}^{(p)}
        \right\rangle
    \right|
    &\leq
    \sqrt{2k2^s}\,\eta,\\
    \left|
        \left\langle
            \mathbf{h}_{\nu_\ell},
            \mathbf{x}+\mathbf{x}^{(p)}
        \right\rangle
    \right|
    &\leq
    2\sqrt{2k2^s}.
\end{aligned}
\end{equation}
Consequently,
\begin{equation}
    \left|
        \|\mathcal{A}(\mathcal{X})\|_2^2
        -\|\mathcal{A}(\mathcal{X}^{(p)})\|_2^2
    \right|
    \leq
    4k2^s\eta.
\end{equation}

\paragraph{Choice of $\eta$ and union bound.}
Choose
\begin{equation}
    \eta
    \coloneqq
    \frac{\varepsilon}{8k2^s}.
\end{equation}
The net approximation error is then at most $\varepsilon/2$, and
\eqref{eq:structured-net-cardinality} gives
\begin{equation}
    |\mathcal{N}_\eta|
    \leq
    \left(
        \frac{16Kd\sqrt{r}\,k2^s}{\varepsilon}
    \right)^{2dr^2}.
\end{equation}

Lemma~\ref{lem:fixed-rip}, applied with accuracy $\varepsilon/2$ to
each $\mathcal{X}^{(p)}\in\mathcal{N}_\eta$, gives
\begin{equation}
    \mathbb{P}\left(
        \left|
            \|\mathcal{A}(\mathcal{X}^{(p)})\|_2^2-1
        \right|
        \geq
        \frac{\varepsilon}{2}
    \right)
    \leq
    2\exp\left(
        -\frac{m\varepsilon^2}{16k2^s}
    \right).
\end{equation}
Taking a union bound over $\mathcal{N}_\eta$, the failure probability
is at most
\begin{equation}
    2
    \left(
        \frac{16Kd\sqrt{r}\,k2^s}{\varepsilon}
    \right)^{2dr^2}
    \exp\left(
        -\frac{m\varepsilon^2}{16k2^s}
    \right).
\end{equation}
By \eqref{eq:uniform-measure-bound}, this probability is at most
$\delta$.

Therefore, with probability at least $1-\delta$, every net point
satisfies
\begin{equation}
    \left|
        \|\mathcal{A}(\mathcal{X}^{(p)})\|_2^2-1
    \right|
    \leq
    \frac{\varepsilon}{2}.
\end{equation}
Combining this with the deterministic approximation estimate gives, for
every $\mathcal{X}\in\mathbb{S}_{r,s,k}$,
\begin{equation}
    \left|
        \|\mathcal{A}(\mathcal{X})\|_2^2-1
    \right|
    \leq
    \frac{\varepsilon}{2}
    +
    \frac{\varepsilon}{2}
    =
    \varepsilon.
\end{equation}
Finally, for any nonzero $\mathcal{X}\in\M_{r,s,k}(C)$, the normalized
tensor $\mathcal{X}/\|\mathcal{X}\|_F$ belongs to
$\mathbb{S}_{r,s,k}$. Homogeneity gives the claimed estimate. The zero
tensor satisfies it trivially.

\end{proof}

\begin{remark}[Comparison with general sparse RIP bounds]
Let $q\coloneqq k2^s$. Then the vectorization of every tensor in
$\M_{r,s,k}(C)$ is $q$-sparse in $\mathbb{R}^N$. For fixed
$\varepsilon$ and $\delta$, using $d=\log_2N$ in
Theorem~\ref{thm:uniform-rip} gives $m=\mathcal{O}\!\left(
    qr^2\log N
    \log\left(\sqrt{r}\,q\log N\right)
\right)$.
For bounded TT rank, this becomes
\begin{equation}
    m
    =
    \mathcal{O}\!\left(
        q\log N\log(q\log N)
    \right).
\end{equation}

For comparison, if the tensor structure is ignored, classical RIP results for
arbitrary $q$-sparse vectors apply. Rudelson and Vershynin showed that, for an
orthogonal matrix with entries bounded by $O(N^{-1/2})$,
\begin{equation}
    m
    =
    \mathcal{O}\!\left(
        q\log N\log^2q\log(q\log N)
    \right)
    =
        \mathcal{O}(q\log^4 N)
\end{equation}
measurements suffice~\cite{rudelson2008sparse}. The normalized
Walsh--Hadamard matrix belongs to this class. For bounded-entry unitary
matrices with independent uniform row sampling, as used here, Haviv and Regev
sharpened the requirement to
\begin{equation}
    m
    =
    \mathcal{O}(q\log^2q\log N)
\end{equation}
measurements~\cite{haviv2017restricted}.
For bounded TT rank, our model-aware estimate removes the factor $\log^2q$
from the Rudelson--Vershynin bound. It does not
uniformly improve on the Haviv--Regev bound: the remaining factors are
$\log(q\log N)$ in our estimate and $\log^2q$ in theirs, so their ordering
depends on how $q$ scales with $N$. Nevertheless, for the smaller family
determined by the TT-rank and interaction constraints, a single-scale covering
argument---without a multiscale chaining construction---already yields a bound
of comparable polylogarithmic order.

\end{remark}

\begin{remark}[Extension to the full order-$s$ interaction space
$\mathcal{V}_s$]

In practice, the number $k$ of active low-order interaction components may be
unknown; one may only know that $\mathcal{X}\in\mathcal{V}_s$ has TT rank at
most $r$. For every $\mathcal{X}\in\mathcal{V}_s$, the support of
$\operatorname{vec}(\mathcal{X})$ is contained in the coordinates indexed by
$S\subseteq[d]$ with $|S|\le s$, and hence has cardinality at most
\[
    D_s=\sum_{\ell=0}^s \binom{d}{\ell}.
\]
Repeating the proof of Theorem~\ref{thm:uniform-rip} with this support bound
in place of $k2^s$, it is sufficient to require
\begin{equation}
    m \geq
    \frac{16D_s}{\varepsilon^2}
    \left[
        \log\frac{2}{\delta}
        +
        2dr^2
        \log\left(
            \frac{8Kd\sqrt{r}\,D_s}{\varepsilon}
        \right)
    \right].
\end{equation}
Using $D_s\leq(ed/s)^s$ from
Remark~\ref{rem:interaction-space-dimension}, it is enough to require
\begin{equation}
    m \geq
    \frac{16}{\varepsilon^2}
    \left(\frac{ed}{s}\right)^s
    \left[
        \log\frac{2}{\delta}
        +
        2dr^2
        \log\left(
            \frac{8Kd\sqrt{r}}{\varepsilon}
            \left(\frac{ed}{s}\right)^s
        \right)
    \right].
\end{equation}
For fixed $s$, this condition remains polynomial in $d$. Thus, extending the
active-component model to the full interaction space $\mathcal{V}_s$ preserves
the qualitative sample-complexity scaling.
\end{remark}

\subsection{Uniqueness of noiseless recovery}

We next explain how a structured RIP implies uniqueness of the noiseless
constrained recovery problem. Let
$\mathcal{X}^\star\in\M_{r,s,k}(C)$ and consider
\begin{equation}
    \label{eq:noiseless-constrained-recovery}
    \widehat{\mathcal{X}}
    \in
    \arg\min_{\mathcal{X}\in\M_{r,s,k}(C)}
    \left\|
        \mathcal{A}(\mathcal{X})-\mathbf{y}
    \right\|_2,
    \qquad
    \mathbf{y}=\mathcal{A}(\mathcal{X}^\star).
\end{equation}

\begin{corollary}[Uniqueness under structured RIP]
\label{cor:uniqueness}
Suppose that $\mathcal{A}$ satisfies the RIP on
$\M_{2r,s,2k}(2C)$ with a constant $\varepsilon<1$; that is,
\begin{equation}
    (1-\varepsilon)\|\mathcal{D}\|_F^2
    \leq
    \|\mathcal{A}(\mathcal{D})\|_2^2
    \leq
    (1+\varepsilon)\|\mathcal{D}\|_F^2
    \qquad
    \text{for every }\mathcal{D}\in\M_{2r,s,2k}(2C).
\end{equation}
Then $\mathcal{X}^\star$ is the unique minimizer of
\eqref{eq:noiseless-constrained-recovery}.
\end{corollary}

\begin{proof}
The ground-truth tensorization is feasible and gives zero residual, so every
minimizer $\widehat{\mathcal{X}}$ satisfies
$\mathcal{A}(\widehat{\mathcal{X}})
=\mathcal{A}(\mathcal{X}^\star)$. Set
$\mathcal{D}\coloneqq
\widehat{\mathcal{X}}-\mathcal{X}^\star$. Since both tensors belong to
$\M_{r,s,k}(C)$, subadditivity of the ranks of their TT unfoldings gives
$\operatorname{rank}_{\mathrm{TT}}(\mathcal{D})\leq2r$. Concatenating their
two cluster-component lists and negating the coefficients of
$\mathcal{X}^\star$ shows that $\mathcal{D}\in\LC_{2k}(s)$. Moreover,
$\|\mathcal{D}\|_F\leq
\|\widehat{\mathcal{X}}\|_F+\|\mathcal{X}^\star\|_F\leq2C$.
Thus, $\mathcal{D}\in\M_{2r,s,2k}(2C)$, and the lower RIP inequality
gives
\begin{equation}
    (1-\varepsilon)\|\mathcal{D}\|_F^2
    \leq
    \|\mathcal{A}(\mathcal{D})\|_2^2
    =
    0.
\end{equation}
Since $\varepsilon<1$, we obtain $\mathcal{D}=0$, and hence
$\widehat{\mathcal{X}}=\mathcal{X}^\star$.
\end{proof}

Consequently, for any $\delta\in(0,1)$, applying
Theorem~\ref{thm:uniform-rip} to $\M_{2r,s,2k}(2C)$ shows that uniqueness
holds with probability at least $1-\delta$, provided
\eqref{eq:uniform-measure-bound} holds with $r$ and $k$ replaced by $2r$
and $2k$.

\section{Analysis of the structure-aware initialization}

Algorithm~\ref{alg:rip-init} proceeds in two stages:
\begin{equation}
\begin{aligned}
    \widetilde{\mathcal{X}}_0
    &=
    \mathcal{P}_{\mathcal{V}_s}
    \mathcal{A}^*\mathcal{A}(\mathcal{X}^\star),\\
    \mathcal{X}_{\mathrm{init}}
    &=
    \operatorname{TT\text{-}SVD}_{\mathbf{r}}
    (\widetilde{\mathcal{X}}_0).
\end{aligned}
\end{equation}
Accordingly, we first bound the projected backprojection error and then
control the deterministic error introduced by TT-SVD truncation.

\subsection{Projected backprojection error}
Let $\mathcal{I}$ denote the identity operator on
$\RR^{2\times\cdots\times2}$. Since
$\M_{r,s,k}(C)\subseteq\mathcal{V}_s$, every
$\mathcal{X}\in\M_{r,s,k}(C)$ satisfies
$\mathcal{P}_{\mathcal{V}_s}(\mathcal{X})=\mathcal{X}$, and hence
\begin{equation}
    \mathcal{P}_{\mathcal{V}_s}
    \mathcal{A}^*\mathcal{A}(\mathcal{X})-\mathcal{X}
    =
    \mathcal{P}_{\mathcal{V}_s}
    (\mathcal{A}^*\mathcal{A}-\mathcal{I})
    \mathcal{P}_{\mathcal{V}_s}(\mathcal{X}).
\end{equation}
Define the compressed error operator
$
    \mathcal{B}_s
    \coloneqq
    \mathcal{P}_{\mathcal{V}_s}
    \bigl(\mathcal{A}^*\mathcal{A}-\mathcal{I}\bigr)
    \mathcal{P}_{\mathcal{V}_s}.
$
The TT-RIP argument used by Qin, Wakin, and
Zhu~\cite{qin2024guaranteed} for the unprojected spectral initializer cannot be
applied directly here. Controlling $\mathcal{B}_s$ requires operator-norm
control of
$\mathcal{P}_{\mathcal{V}_s}\mathcal{A}^*\mathcal{A}
\mathcal{P}_{\mathcal{V}_s}$ on the full space $\mathcal{V}_s$, whereas
Theorem~\ref{thm:uniform-rip} gives RIP only on
$\M_{r,s,k}(C)\subseteq\mathcal{V}_s$. Applying the same argument would
therefore require an RIP on the larger space $\mathcal{V}_s$, which is not
provided by Theorem~\ref{thm:uniform-rip}. We instead analyze
$\mathcal{B}_s$ directly using matrix concentration.
It follows that
\begin{equation}
\label{eq:model-backprojection-by-operator}
    \sup_{\substack{\mathcal{X}\in\M_{r,s,k}(C)\\
                    \mathcal{X}\neq0}}
    \frac{
        \|\mathcal{P}_{\mathcal{V}_s}
        \mathcal{A}^*\mathcal{A}(\mathcal{X})-\mathcal{X}\|_F
    }{
        \|\mathcal{X}\|_F
    }
    \leq
    \|\mathcal{B}_s\|_{F\to F},
\end{equation}
where $\|\cdot\|_{F\to F}$ is the operator norm on $\mathcal{V}_s$
induced by the Frobenius norm. We first establish a high-probability
bound for this operator norm using matrix Bernstein's inequality.

\begin{lemma}[Compressed error operator bound]
\label{lem:compressed-error-operator}
For every $\varepsilon\in(0,1)$,
\begin{equation}
\label{eq:compressed-error-operator-tail}
    \mathbb{P}\left(
        \|\mathcal{B}_s\|_{F\to F}\geq\varepsilon
    \right)
    \leq
    2D_s\exp\left(-\frac{3m\varepsilon^2}{8D_s}\right).
\end{equation}
\end{lemma}

\begin{proof}
Under vectorization, $\mathcal{B}_s$ is represented on $V_s$ by the
self-adjoint matrix
\begin{equation}
    \mathbf{B}_s
    \coloneqq
    P_{V_s}
    (\mathbf{A}^*\mathbf{A}-\mathbf{I}_N)
    P_{V_s}.
\end{equation}
Since vectorization is an isometry,
$\|\mathcal{B}_s\|_{F\to F}=\|\mathbf{B}_s\|_2$.
For each $\ell$, define
\begin{equation}
    \mathbf{u}_\ell
    \coloneqq
    P_{V_s}\mathbf{h}_{\nu_\ell}
    \in V_s,
    \qquad
    \mathbf{Z}_\ell
    \coloneqq
    \mathbf{u}_\ell\mathbf{u}_\ell^*-P_{V_s}.
\end{equation}
By \eqref{eq:hadamard-row-measurement},
\begin{equation}
\label{eq:compressed-operator-matrix}
    \mathbf{B}_s
    =
    \frac{1}{m}\sum_{\ell=1}^m\mathbf{Z}_\ell.
\end{equation}
The matrices $\mathbf{Z}_1,\ldots,\mathbf{Z}_m$ are independent and
self-adjoint.

Uniform sampling and Hadamard orthogonality give
\begin{equation}
    \mathbb{E}[\mathbf{u}_\ell\mathbf{u}_\ell^*]
    =
    P_{V_s}
    \left(
        \frac{1}{N}\sum_{j=1}^N
        \mathbf{h}_j
        \mathbf{h}_j^*
    \right)
    P_{V_s}
    =P_{V_s}.
\end{equation}
Hence $\mathbb{E}\mathbf{Z}_\ell=0$. Since $P_{V_s}$ retains $D_s$
coordinates and every entry of $\mathbf{h}_{\nu_\ell}$ has magnitude one,
\begin{equation}
    \|\mathbf{u}_\ell\|_2^2=D_s.
\end{equation}
On $V_s$, the eigenvalues of $\mathbf{Z}_\ell$ are $D_s-1$ in the
direction of $\mathbf{u}_\ell$ and $-1$ on its orthogonal complement.
Thus, $\|\mathbf{Z}_\ell\|_2\leq D_s$. Furthermore, since
$P_{V_s}\mathbf{u}_\ell=\mathbf{u}_\ell$ and $P_{V_s}^2=P_{V_s}$,
\begin{equation}
\begin{aligned}
    \mathbf{Z}_\ell^2
    &=
    (\mathbf{u}_\ell\mathbf{u}_\ell^*-P_{V_s})^2\\
    &=
    (D_s-2)\mathbf{u}_\ell\mathbf{u}_\ell^*+P_{V_s},\\
    \mathbb{E}\mathbf{Z}_\ell^2
    &=
    (D_s-2)P_{V_s}+P_{V_s}\\
    &=
    (D_s-1)P_{V_s}.
\end{aligned}
\end{equation}
Thus, the variance parameter is
$\|\sum_{\ell=1}^m\mathbb{E}\mathbf{Z}_\ell^2\|_2
=m(D_s-1)$.
The self-adjoint matrix Bernstein inequality
\cite{Tropp2012UserFriendly}, with threshold $m\varepsilon$, yields
\begin{equation}
\begin{aligned}
    \mathbb{P}
    \left(
        \|\mathcal{B}_s\|_{F\to F}\geq\varepsilon
    \right)
    &\leq
    2D_s\exp\left(
        -\frac{m\varepsilon^2}
        {2(D_s-1)+(2/3)D_s\varepsilon}
    \right)\\
    &\leq
    2D_s\exp\left(
        -\frac{3m\varepsilon^2}{8D_s}
    \right),
\end{aligned}
\end{equation}
where the last inequality uses $\varepsilon<1$.
\end{proof}

Combining Lemma~\ref{lem:compressed-error-operator} with
\eqref{eq:model-backprojection-by-operator} gives the central
backprojection guarantee.

\begin{theorem}[Uniform projected backprojection bound]
\label{thm:uniform-backprojection}
Let $\varepsilon,\delta\in(0,1)$. If
\begin{equation}
\label{eq:uniform-backprojection-sample-complexity}
    m
    \geq
    \frac{8D_s}{3\varepsilon^2}
    \log\left(\frac{2D_s}{\delta}\right),
\end{equation}
then, with probability at least $1-\delta$,
\begin{equation}
    \left\|
        \mathcal{P}_{\mathcal{V}_s}
        \mathcal{A}^*\mathcal{A}(\mathcal{X})
        -\mathcal{X}
    \right\|_F
    \leq
    \varepsilon\|\mathcal{X}\|_F
    \qquad
    \text{for every }\mathcal{X}\in\M_{r,s,k}(C).
\end{equation}
\end{theorem}

\subsection{TT-SVD truncation and initialization guarantee}
The second initialization stage imposes the target TT-rank by truncated
TT-SVD. The following deterministic lemma transfers any
pre-truncation error bound to the truncated tensor.

\begin{lemma}[TT-SVD postprocessing error]
\label{lem:tt-svd-postprocessing}
Let $\mathcal{X}$ be a $d$-way tensor satisfying
$\operatorname{rank}_{\mathrm{TT}}(\mathcal{X})\leq\mathbf{r}$, and let
$\widetilde{\mathcal{X}}$ be any perturbation of $\mathcal{X}$. Define
\begin{equation}
    \widehat{\mathcal{X}}
    \coloneqq
    \operatorname{TT\text{-}SVD}_{\mathbf{r}}
    (\widetilde{\mathcal{X}}).
\end{equation}
Then
\begin{equation}
    \|\widehat{\mathcal{X}}-\mathcal{X}\|_F
    \leq
    \left(1+\sqrt{d-1}\right)
    \|\widetilde{\mathcal{X}}-\mathcal{X}\|_F.
\end{equation}
\end{lemma}

\begin{proof}
The triangle inequality and TT-SVD quasi-optimality
\cite[Corollary~2.4]{Oseledets2011} give
\begin{equation}
\begin{aligned}
    \|\widehat{\mathcal{X}}-\mathcal{X}\|_F
    &\leq
    \|\widehat{\mathcal{X}}-\widetilde{\mathcal{X}}\|_F
    +\|\widetilde{\mathcal{X}}-\mathcal{X}\|_F\\
    &\leq
    \sqrt{d-1}
    \inf_{\operatorname{rank}_{\mathrm{TT}}(\mathcal{Z})\leq\mathbf{r}}
    \|\widetilde{\mathcal{X}}-\mathcal{Z}\|_F
    +\|\widetilde{\mathcal{X}}-\mathcal{X}\|_F\\
    &\leq
    \left(1+\sqrt{d-1}\right)
    \|\widetilde{\mathcal{X}}-\mathcal{X}\|_F.
\end{aligned}
\end{equation}
The last inequality uses the admissible choice $\mathcal{Z}=\mathcal{X}$.
\end{proof}

Combining Theorem~\ref{thm:uniform-backprojection} with
Lemma~\ref{lem:tt-svd-postprocessing} gives the final initialization
guarantee.

\begin{theorem}[Uniform initialization error bound]
\label{thm:initialization-error}
Let $\varepsilon,\delta\in(0,1)$ and set
$\alpha_d\coloneqq 1+\sqrt{d-1}$. If
\begin{equation}
    \label{eq:initialization-sample-complexity}
    m
    \geq
    \frac{8D_s\alpha_d^2}{3\varepsilon^2}
    \log\left(\frac{2D_s}{\delta}\right),
\end{equation}
then, with probability at least $1-\delta$, simultaneously for every
$\mathcal{X}\in\M_{r,s,k}(C)$ and every target rank vector $\mathbf{r}$
satisfying $\operatorname{rank}_{\mathrm{TT}}(\mathcal{X})\leq\mathbf{r}$,
the initializer
\begin{equation}
    \mathcal{X}_{\mathrm{init}}
    =
    \operatorname{TT\text{-}SVD}_{\mathbf{r}}
    \left(
        \mathcal{P}_{\mathcal{V}_s}
        \mathcal{A}^*\mathcal{A}(\mathcal{X})
    \right)
\end{equation}
satisfies
\begin{equation}
    \|\mathcal{X}_{\mathrm{init}}-\mathcal{X}\|_F
    \leq
    \varepsilon\|\mathcal{X}\|_F.
\end{equation}
\end{theorem}

\begin{proof}
Apply Theorem~\ref{thm:uniform-backprojection} with accuracy
$\varepsilon/\alpha_d$ and then Lemma~\ref{lem:tt-svd-postprocessing}.
Under \eqref{eq:initialization-sample-complexity}, with probability at
least $1-\delta$,
\begin{equation}
    \|\mathcal{X}_{\mathrm{init}}-\mathcal{X}\|_F
    \leq
    \alpha_d
    \left\|
        \mathcal{P}_{\mathcal{V}_s}
        \mathcal{A}^*\mathcal{A}(\mathcal{X})
        -\mathcal{X}
    \right\|_F
    \leq
    \varepsilon\|\mathcal{X}\|_F.
\end{equation}
The estimate holds simultaneously for every admissible $\mathcal{X}$ and
$\mathbf{r}$.
\end{proof}

\begin{remark}[Scaling of the initialization guarantee]
For $1\leq s<d$, the bound $D_s\leq(ed/s)^s$ shows that it is sufficient
to take
\begin{equation}
    m
    \geq
    \frac{8\alpha_d^2}{3\varepsilon^2}
    \left(\frac{ed}{s}\right)^s
    \log\left[
        \frac{2}{\delta}
        \left(\frac{ed}{s}\right)^s
    \right].
\end{equation}
For fixed $s$, $\varepsilon$, and $\delta$, we have
$\alpha_d^2=\mathcal{O}(d)$; hence, the sufficient sample size satisfies
\begin{equation}
    m=\mathcal{O}\left(d^{s+1}\log d\right).
\end{equation}
Thus, for fixed interaction order $s$, the required number of
measurements grows polynomially with the tensor order $d$, rather than
exponentially through the ambient dimension $N=2^d$.
\end{remark}

\section{Experiments}
We test three aspects of the proposed framework: restricted isometry on the
structured model class, the accuracy of the structure-aware initialization, and
its effect on the full ALS recovery workflow.

\paragraph{Common experimental setup.}
All experiments use synthetic data from the same sliding-window interaction
model so that the sparse low-order interaction structure and TT ranks are
controlled.

Fix an interaction order \(s\), a number of active components \(k\), and a
step size \(p\). For \(t=1,\ldots,k\), define the sliding-window mode sets
\begin{equation}
    a_t\coloneqq 1+(t-1)p,
    \qquad
    I_t\coloneqq\{a_t,\ldots,a_t+s-1\},
\end{equation}
where \(a_t+s-1\leq d\). We then construct the target tensor as a sum of
rank-one cluster components associated with these windows:
\begin{equation}
\label{eq:sliding-window-model}
\begin{aligned}
    \mathcal{X}^\star&\coloneqq\sum_{t=1}^k \mathcal{X}_{I_t},
    &\mathcal{X}_{I_t}&\coloneqq
    \mathbf{u}_1^{(t)}\circ\cdots\circ\mathbf{u}_d^{(t)},\\
    \mathbf{u}_j^{(t)}&\coloneqq
    \begin{cases}
        \mathbf{v}_j^{(t)}, & j\in I_t,\\
        (1,0)^{\mathsf T}, & j\notin I_t.
    \end{cases}
\end{aligned}
\end{equation}
Each active vector is sampled independently as
\(\mathbf{v}_j^{(t)}=\mathbf{z}_j^{(t)}/\|\mathbf{z}_j^{(t)}\|_2\), where
\(\mathbf{z}_j^{(t)}\sim\mathcal N(0,\mathbf{I}_2)\).
By the definition of $\LC_k(s)$, $\mathcal{X}_{I_t}$ is a cluster component
involving only the modes in $I_t$.
Thus, \(\mathcal{X}^\star\in\LC_k(s)\subseteq\mathcal{V}_s\) and
\(\operatorname{rank}_{\mathrm{TT}}(\mathcal{X}^\star)\leq k\). In practice, the
TT ranks can be smaller because only active windows crossing a given
tensor-train cut contribute to the corresponding unfolding rank. For each tensor
\(\mathcal{X}\), let
\(r_{\max}(\mathcal{X})\coloneqq\max_j r_j(\mathcal{X})\).
For each fixed sliding-window pattern, \(r_{\max}(\mathcal{X})\) appeared
empirically constant across the generated tensors: its sample mean and sample
maximum coincided within each test set.

\subsection{Tests for the RIP}
\label{subsec:rip-experiments}
Theorem~\ref{thm:uniform-rip} takes a supremum over the full model class and
gives a probabilistic guarantee for the random operator. We investigate it
empirically using a finite-test-set approximation to the supremum and repeated
independent draws of the operator.

We draw \(100\) independent operators whose rows are sampled uniformly with
replacement. For each operator, we generate an independent test set
\(\mathcal T\) of \(200\) tensors and estimate its largest relative distortion
by
\begin{equation}
    \widehat{\delta}_{\mathcal T}(\mathcal A)
    \coloneqq
    \max_{\mathcal X\in\mathcal T}\left|
        \frac{\|\mathcal A(\mathcal X)\|_2^2}{\|\mathcal X\|_F^2}-1
    \right|.
\end{equation}
We report the mean of \(\widehat{\delta}_{\mathcal T}(\mathcal A)\) over the
\(100\) operators. All operator applications use the fast Walsh--Hadamard
transform defined in Section~\ref{subsec:notation}.

\begin{figure}[!htbp]
    \centering
    \begin{subfigure}[t]{0.47\textwidth}
        \centering
        \includegraphics[width=\linewidth]{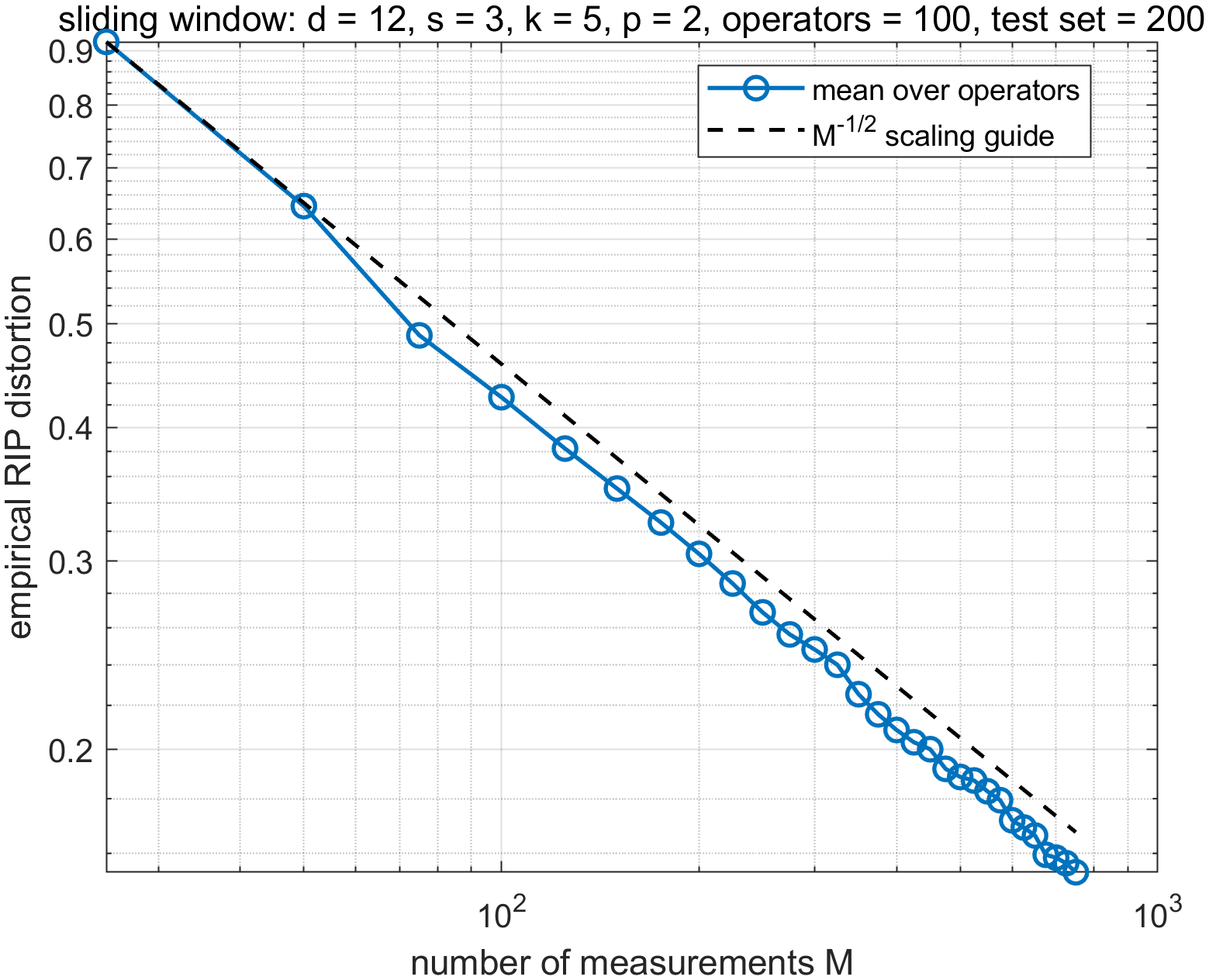}
        \caption{Measurements \(m\).}
        \label{fig:rip-measurement-scaling}
    \end{subfigure}
    \hfill
    \begin{subfigure}[t]{0.47\textwidth}
        \centering
        \includegraphics[width=\linewidth]{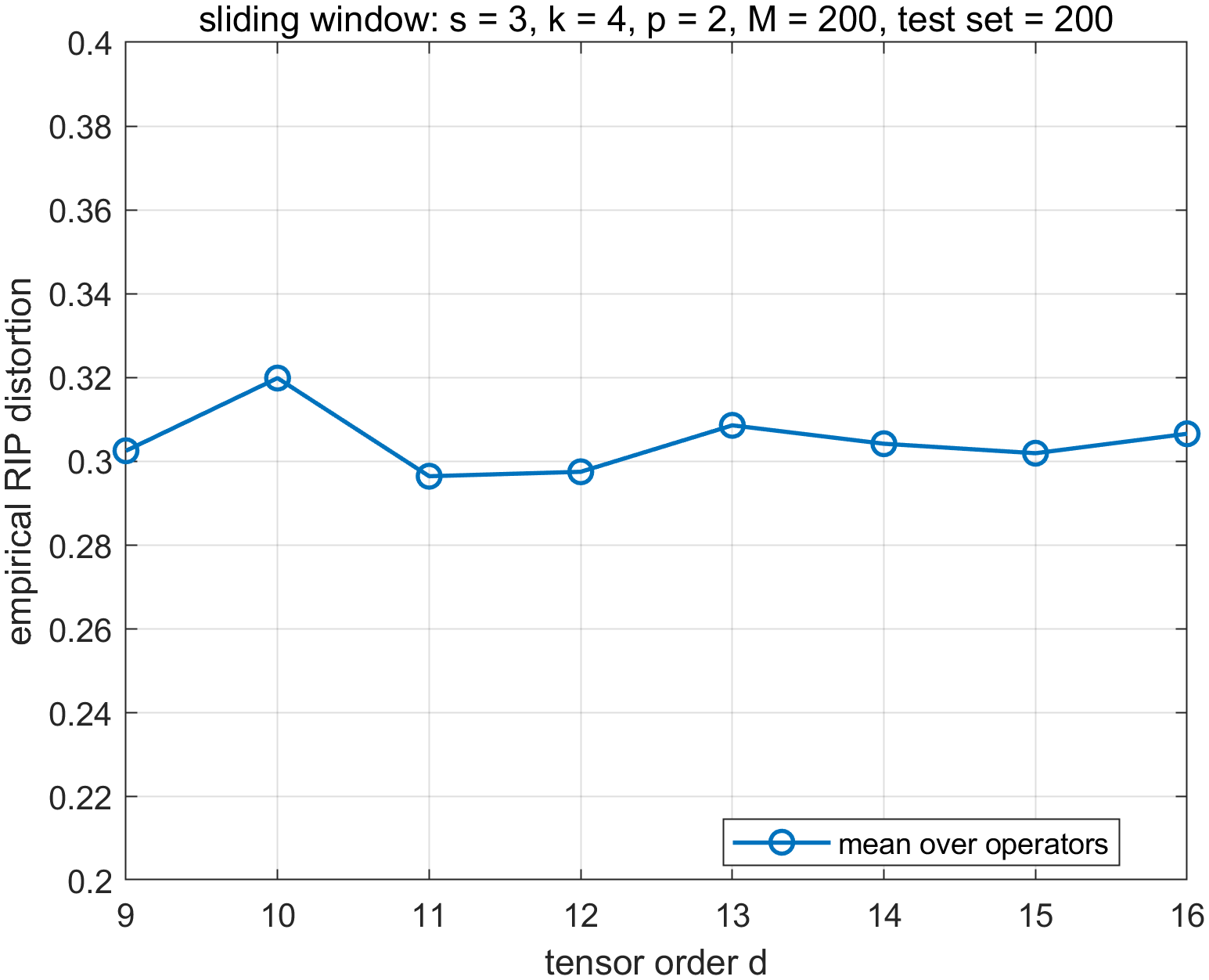}
        \caption{Tensor order \(d\).}
        \label{fig:rip-dimension-scaling}
    \end{subfigure}
    \par\vspace{0.15em}
    \begin{subfigure}[t]{0.47\textwidth}
        \centering
        \includegraphics[width=\linewidth]{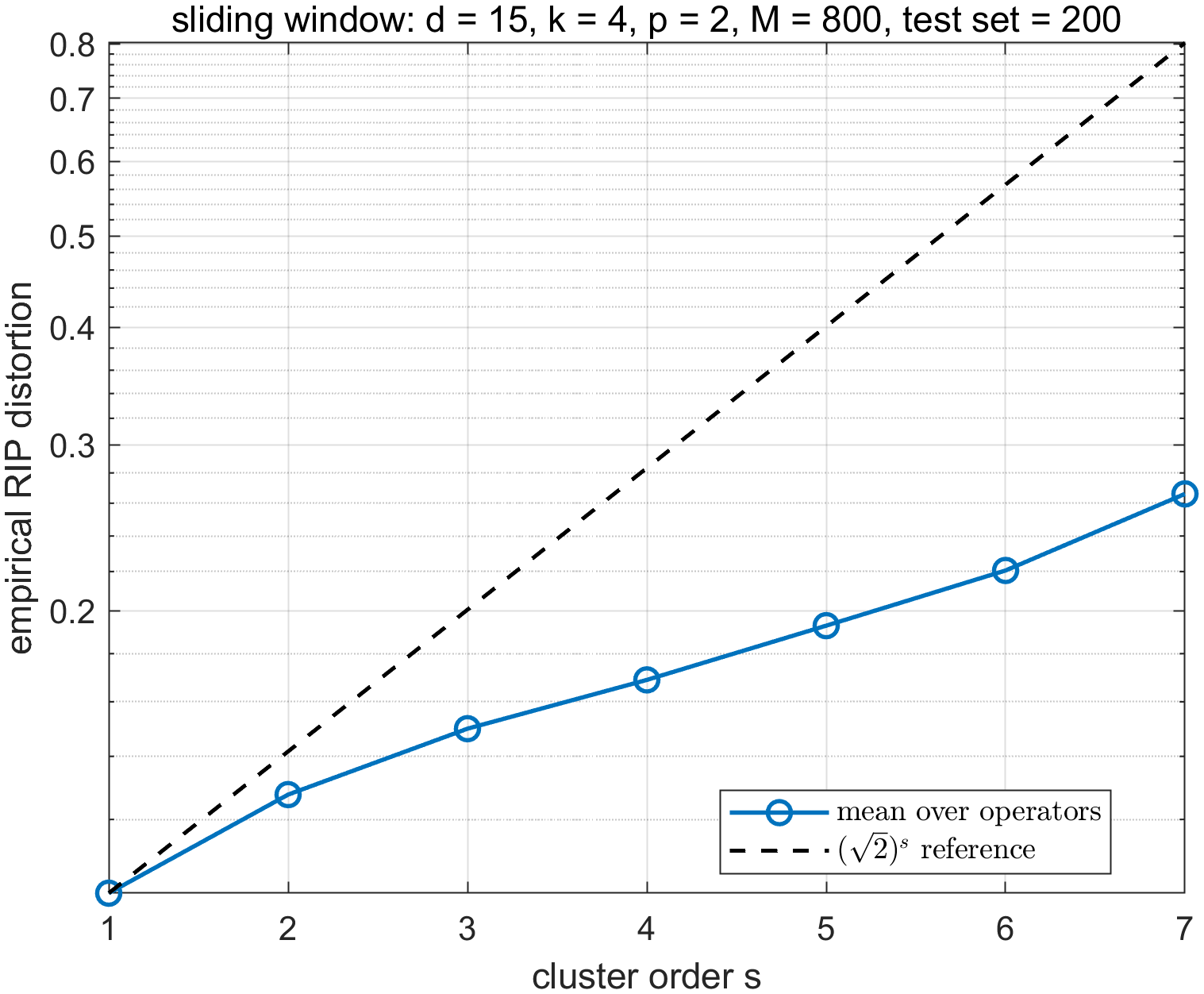}
        \caption{Interaction order \(s\).}
        \label{fig:rip-interaction-order-scaling}
    \end{subfigure}
    \hfill
    \begin{subfigure}[t]{0.47\textwidth}
        \centering
        \includegraphics[width=\linewidth]{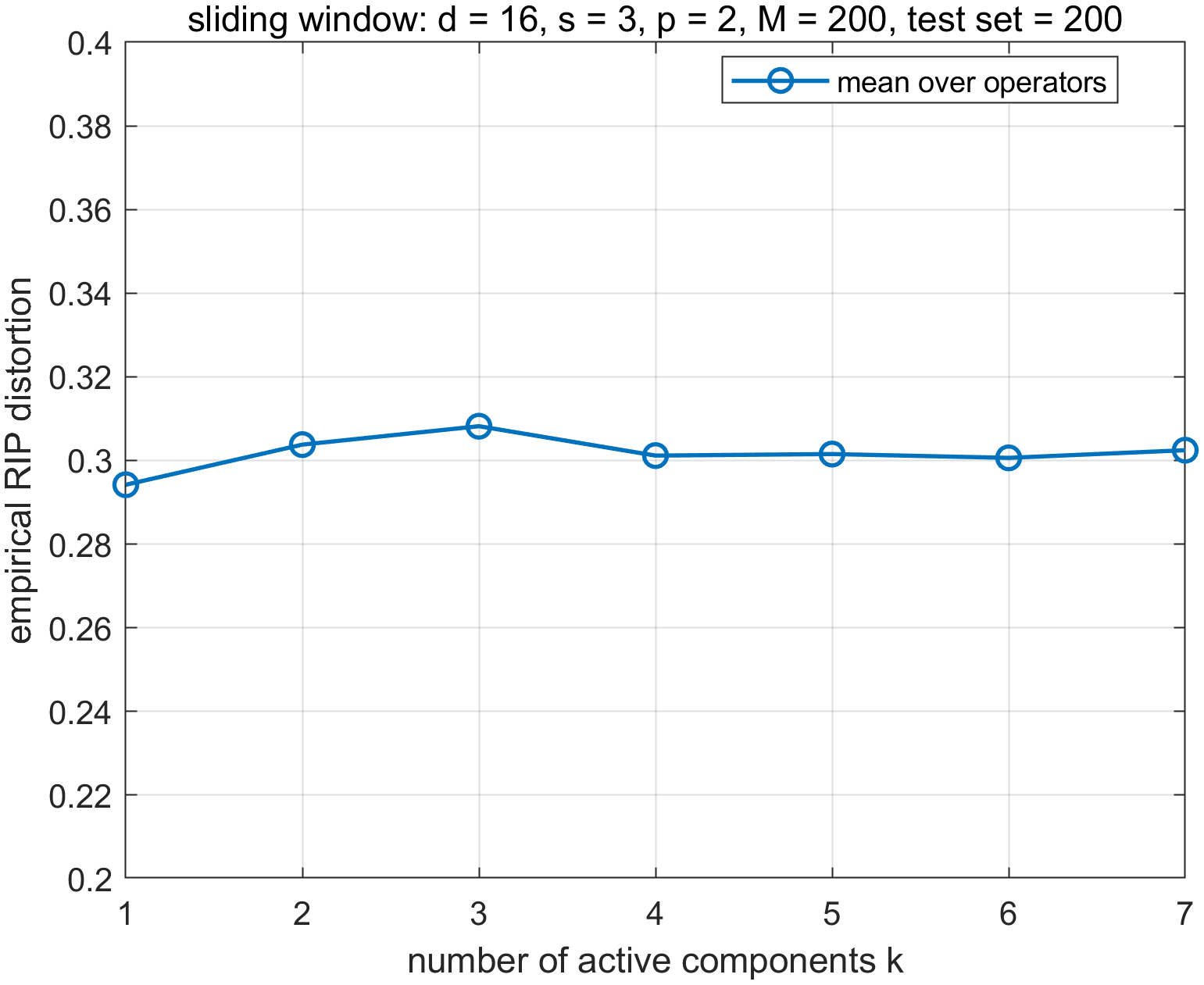}
        \caption{Active components \(k\).}
        \label{fig:rip-active-components-scaling}
    \end{subfigure}

    \caption{Empirical RIP estimates under sweeps of
    (a) \(m=25,50,\ldots,750\), (b) \(d=9,\ldots,16\),
    (c) \(s=1,\ldots,7\), and (d) \(k=1,\ldots,7\).}
    \label{fig:rip-sweeps}
\end{figure}
\FloatBarrier

To extract qualitative predictions from Theorem~\ref{thm:uniform-rip}, we vary
one parameter at a time. Fixing the failure probability and ignoring logarithmic factors and the dependence on the TT-rank bound \(r\), the sample-complexity bound in \eqref{eq:uniform-measure-bound} suggests the following heuristic scalings:
\begin{equation}
    \varepsilon_{\mathrm{RIP}}
    \propto
    \begin{cases}
        m^{-1/2}, & \text{as the measurement budget }m\text{ varies},\\
        d^{1/2}, & \text{as the tensor order }d\text{ varies},\\
        (\sqrt{2})^s, & \text{as the interaction order }s\text{ varies},\\
        k^{1/2}, & \text{as the number of active components }k\text{ varies}.
    \end{cases}
\end{equation}
The reference curves in panels~(a) and~(c) are scaling guides, not
numerical upper bounds: each is normalized to the first empirical point of the
corresponding sweep and indicates only the predicted dependence on the varied
parameter.

Figure~\ref{fig:rip-sweeps}(a) varies the measurement budget
\(m=25,50,\ldots,750\) and fixes \((d,s,k,p)=(12,3,5,2)\). The empirical
distortion closely follows the Monte Carlo rate \(m^{-1/2}\).

Panel~(b) varies the tensor order \(d=9,\ldots,16\) and fixes
\((s,k,p,m)=(3,4,2,200)\). The distortion remains between approximately
\(0.30\) and \(0.32\), with no visible \(d^{1/2}\) growth; the observed
dependence on \(d\) is therefore milder than the leading worst-case
prediction, suggesting that the current theory may be conservative.

Panel~(c) varies the interaction order \(s=1,\ldots,7\) and fixes
\((d,k,p,m)=(15,4,2,800)\). The distortion increases with \(s\), as predicted,
but much more slowly than the \((\sqrt{2})^s\) guide. The theorem uses the
generic sparsity estimate
\(|\operatorname{supp}(\operatorname{vec}(\mathcal X))|\leq
k2^s\), obtained by adding the \(2^s\) possible cluster-basis coordinates of
each component. Neighboring windows share some of these coordinates, so the
support of their sum can be substantially smaller than \(k2^s\). This smaller
effective sparsity may partly explain the milder observed growth.

Panel~(d) varies the number of active components \(k=1,\ldots,7\) and fixes
\((d,s,p,m)=(16,3,2,200)\). The distortion remains close to \(0.30\), with no
visible \(k^{1/2}\) growth; the observed dependence on \(k\) is again more
favorable than the leading worst-case prediction.

Together, the four sweeps support Theorem~\ref{thm:uniform-rip} empirically,
while indicating that its dependence on \(d\), \(s\), and \(k\) is
conservative for the sampled sliding-window family. The experiment replaces
the supremum over the full model class by a maximum over a finite
sliding-window test set, so these results support the predicted scaling but do
not verify the uniform guarantee.

\subsection{Initialization error and scaling}

To isolate the effect of each step in Algorithm~\ref{alg:rip-init}, we compare
three estimates against the ground-truth tensor $\mathcal{X}^\star$: the
adjoint backprojection $\mathcal{A}^*\mathbf{y}$, its projection onto
$\mathcal{V}_s$, and the TT-SVD truncation of that projection to rank
$\mathbf{r}$. We also include the standard spectral initializer, obtained by
applying the same TT-SVD truncation directly to
$\mathcal{A}^*\mathbf{y}$, as a benchmark that omits the model projection. For
every estimate $\mathcal{Z}$,
we report the relative error
$e(\mathcal{Z})=\|\mathcal{Z}-\mathcal{X}^\star\|_F/
\|\mathcal{X}^\star\|_F$. Because $\mathcal{X}^\star$ is a sum of $k$
rank-one cluster components, its TT ranks are at most $k$. We therefore use
the uniform TT-SVD rank cap $r_j=k$, $j=1,\ldots,d-1$. Here, $k$ denotes the
number of rank-one cluster components in \eqref{eq:sliding-window-model}. As in
the RIP experiments, we pair each of $100$ independent operators with $200$
independently generated tensors and report the mean, over operators, of the
maximum relative error on each test set. The measurement operator and its
adjoint are implemented via the fast Walsh--Hadamard transform, and
$\mathcal{P}_{\mathcal{V}_s}$ is evaluated as a coordinate mask.

In each experiment, we vary one parameter while holding the others fixed: the
measurement budget $m\in\{150,175,\ldots,800\}$ with
$(d,s,k,p)=(10,3,3,2)$; the tensor order $d\in\{9,\ldots,14\}$ with
$(s,k,p,m)=(3,3,2,500)$; and the interaction order
$s\in\{1,\ldots,7\}$ with $(d,k,p,m)=(12,4,1,800)$. Here, $p$ denotes the
step size between consecutive sliding windows.

To compare the observed trends with
Theorem~\ref{thm:uniform-backprojection} and
Lemma~\ref{lem:tt-svd-postprocessing}, we use the scaling guides
\begin{equation}
\begin{aligned}
g_{\mathrm{bp}}(d,s,m)
&\coloneqq
\left[\frac{D_s}{m}\log\!\left(\frac{2D_s}{\delta}\right)\right]^{1/2},\\
g_{\mathrm{init}}(d,s,m)
&\coloneqq
(1+\sqrt{d-1})g_{\mathrm{bp}}(d,s,m),
\end{aligned}
\end{equation}
with \(\delta=0.05\), for the projected backprojection and post-TT-SVD
errors, respectively.
When only \(m\) varies, both guides reduce to the
\(m^{-1/2}\) rate. Each guide is normalized independently to the first point
of its empirical curve, so the comparison concerns scaling rather than the
absolute constants in the bounds.

\begin{figure}[!htbp]
    \centering
    \begin{subfigure}[t]{0.48\textwidth}
        \centering
        \includegraphics[width=\linewidth]{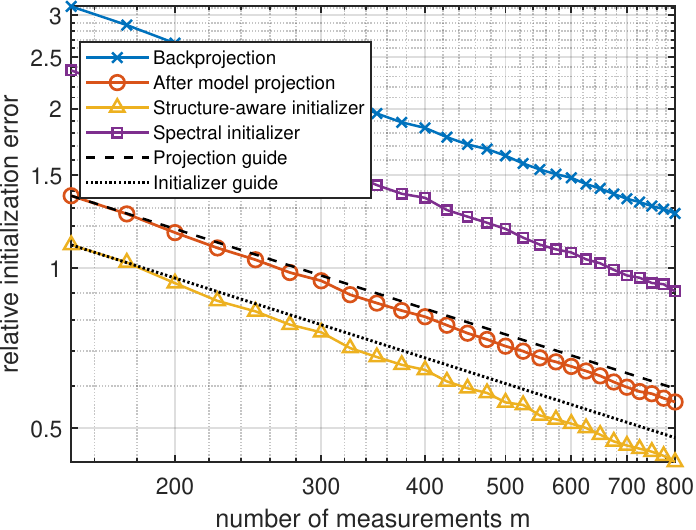}
        \caption{Measurements \(m\).}
        \label{fig:init-measurement-scaling}
    \end{subfigure}
    \hfill
    \begin{subfigure}[t]{0.48\textwidth}
        \centering
        \includegraphics[width=\linewidth]{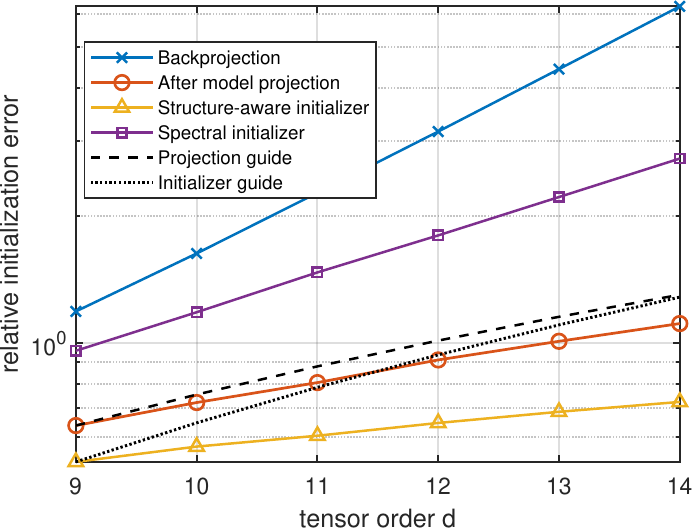}
        \caption{Tensor order \(d\).}
        \label{fig:init-dimension-scaling}
    \end{subfigure}
    \par\vspace{0.2em}
    \begin{subfigure}[t]{0.48\textwidth}
        \centering
        \includegraphics[width=\linewidth]{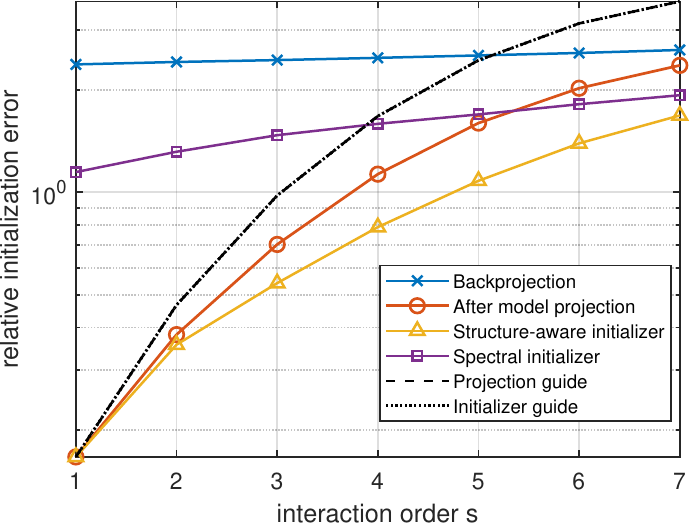}
        \caption{Interaction order \(s\).}
        \label{fig:init-interaction-order-scaling}
    \end{subfigure}

    \caption{Empirical initialization errors at the three successive stages
    and for the spectral benchmark.}
    \label{fig:init-scaling}
\end{figure}
\FloatBarrier

Figure~\ref{fig:init-scaling} supports three conclusions. First, the projected
backprojection and the post-TT-SVD initializer are consistent with the scaling
predicted by Theorems~\ref{thm:uniform-backprojection}
and~\ref{thm:initialization-error}, respectively. Second, the
consistent ordering
\[
e(\mathcal A^*\mathbf y)
>
e(\mathcal P_{\mathcal V_s}\mathcal A^*\mathbf y)
>
e\!\left(\operatorname{TT\text{-}SVD}_{\mathbf r}
(\mathcal P_{\mathcal V_s}\mathcal A^*\mathbf y)\right)
\]
shows that both the interaction-space projection and the TT-SVD truncation
reduce the error, supporting the design of the proposed initializer.
For the TT-SVD step, however, Lemma~\ref{lem:tt-svd-postprocessing} guarantees
only that the relative post-TT-SVD error is bounded by
\((1+\sqrt{d-1})e(\mathcal P_{\mathcal V_s}\mathcal A^*\mathbf y)\).
In these experiments, TT-SVD therefore acts as a denoising step. Analyzing the
spectra of the TT unfoldings of
\(\mathcal P_{\mathcal V_s}\mathcal A^*\mathbf y\) may lead to a sharper
post-TT-SVD error bound. Third, the structure-aware initializer
has lower error than
\(\operatorname{TT\text{-}SVD}_{\mathbf r}(\mathcal A^*\mathbf y)\)
throughout, showing that the additional model projection improves on the
standard spectral initializer in this setting.

\subsection{Initialization in the full ALS recovery scheme}

The preceding experiments assess the initializer before optimization; we now
test it within the full fixed-rank ALS scheme of Section~\ref{sec:iter-opt}.
Within each paired trial, the methods differ only in initialization: they use
the same ground-truth tensor, Walsh--Hadamard measurements, oracle TT ranks,
and fixed-rank ALS recovery algorithm. The proposed structure-aware
initializer is
$
\operatorname{TT\text{-}SVD}_{\mathbf r}\!\left(
\mathcal{P}_{\mathcal{V}_s}\mathcal{A}^*\mathbf{y}
\right).
$
We compare it with the standard unprojected spectral initializer
$
\operatorname{TT\text{-}SVD}_{\mathbf r}
\!\left(\mathcal{A}^*\mathbf{y}\right),
$
which isolates the effect of the low-order projection.

The two uninformative baselines are a random TT tensor generated by
\texttt{tt\_random} with the oracle ranks and a near-zero initializer. Because
an exact zero tensor has no nondegenerate prescribed-rank TT-SVD factors, the
latter is formed by scaling a Gaussian vector to
\(10^{-6}\|\mathcal{X}^\star\|_F\) and truncating it to the oracle ranks. For each
parameter value, we run \(100\) independent tensor--operator trials, each for
at most \(100\) ALS sweeps. Success means a final relative error below
\(10^{-5}\). We report the success rate, mean final relative error, and mean
first-passage sweep count, assigning a value of \(101\) to trials that do not
reach the tolerance within \(100\) sweeps.

\paragraph{Varying the tensor order.}
We first vary the tensor order from \(d=9\) to \(16\), while fixing the
interaction order at \(s=3\), the number of active components at \(k=3\), the
window step size at \(p=2\), and the measurement budget at \(m=400\). Thus, the ambient
dimension grows from \(2^9\) to \(2^{16}\) while the local interaction
structure and measurement budget remain fixed.

\begin{figure}[!htbp]
    \centering

    \begin{subfigure}[t]{0.32\textwidth}
        \centering
        \includegraphics[width=\linewidth]
        {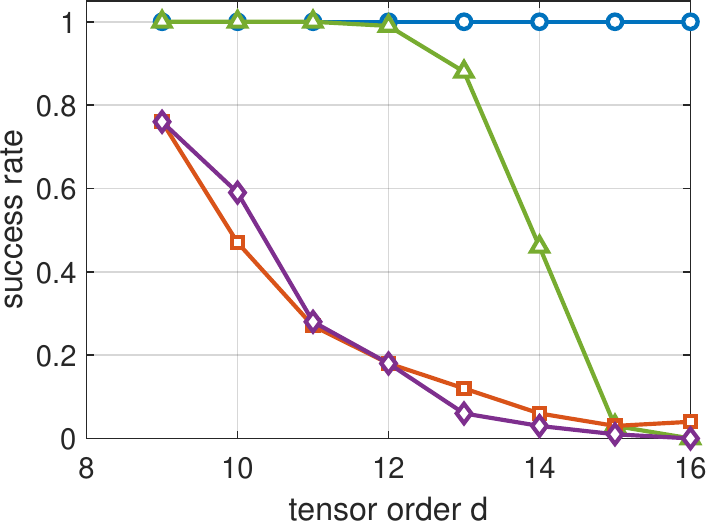}
        \caption{Recovery success rate.}
        \label{fig:recovery-dimension-success}
    \end{subfigure}
    \hfill
    \begin{subfigure}[t]{0.32\textwidth}
        \centering
        \includegraphics[width=\linewidth]
        {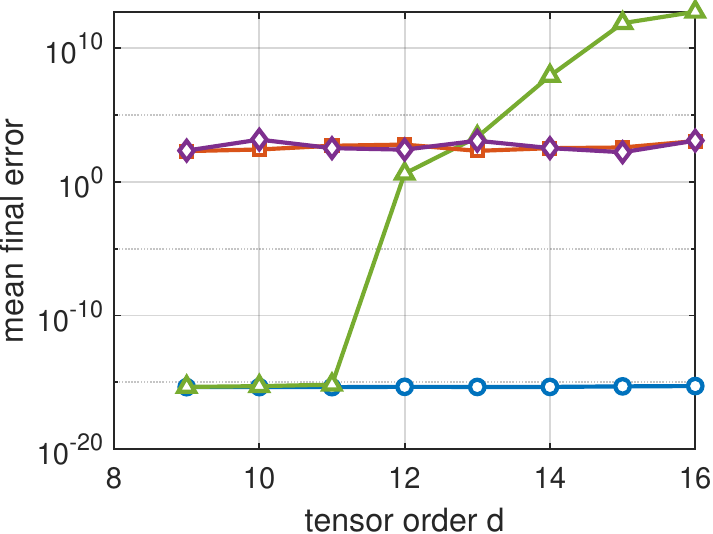}
        \caption{Mean final relative error.}
        \label{fig:recovery-dimension-error}
    \end{subfigure}
    \hfill
    \begin{subfigure}[t]{0.32\textwidth}
        \centering
        \includegraphics[width=\linewidth]
        {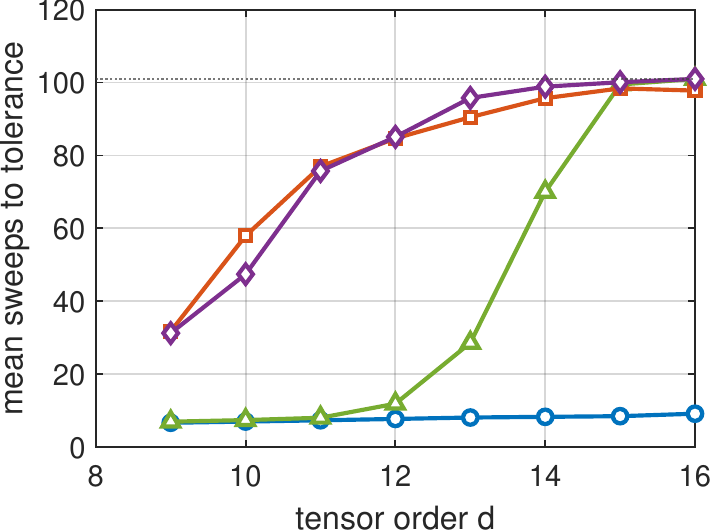}
        \caption{Mean sweeps to tolerance.}
        \label{fig:recovery-dimension-iterations}
    \end{subfigure}

    \caption{Full ALS recovery versus tensor order for \(m=400\), based on
    \(100\) trials. Initializers: structure-aware (blue circles), spectral
    (green triangles), random TT (orange squares), and near-zero (purple
    diamonds).}
    \label{fig:recovery-dimension}
\end{figure}

At fixed \(m=400\), the structure-aware initializer recovers every trial at
every tested tensor order, with near-machine-precision final error and only a
modest increase in sweep count. In contrast, the spectral initializer becomes
unreliable as \(d\) grows and fails all trials at \(d=16\); the uninformative
baselines deteriorate as well. Since the two spectral-type methods differ only
by \(\mathcal P_{\mathcal V_s}\), their separation isolates the benefit of
removing off-model components before TT-SVD.

\paragraph{Varying the number of measurements.}
We next vary the measurement budget as \(m=100,150,\ldots,1200\), while fixing
the tensor order at \(d=12\), the interaction order at \(s=3\), the number of
active components at \(k=3\), and the window step size at \(p=2\).

\begin{figure}[!htbp]
    \centering

    \begin{subfigure}[t]{0.32\textwidth}
        \centering
        \includegraphics[width=\linewidth]
        {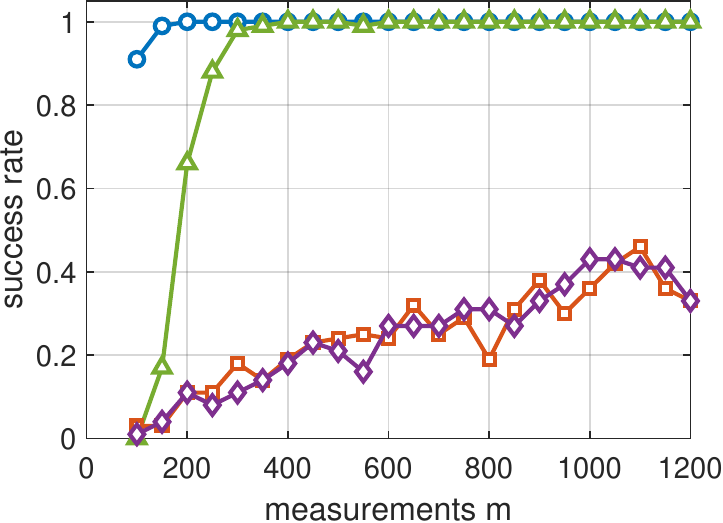}
        \caption{Recovery success rate.}
        \label{fig:recovery-measurement-success}
    \end{subfigure}
    \hfill
    \begin{subfigure}[t]{0.32\textwidth}
        \centering
        \includegraphics[width=\linewidth]
        {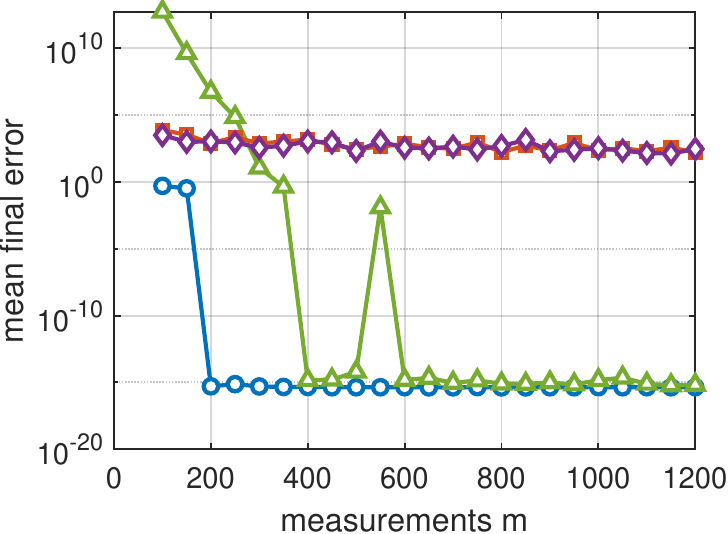}
        \caption{Mean final relative error.}
        \label{fig:recovery-measurement-error}
    \end{subfigure}
    \hfill
    \begin{subfigure}[t]{0.32\textwidth}
        \centering
        \includegraphics[width=\linewidth]
        {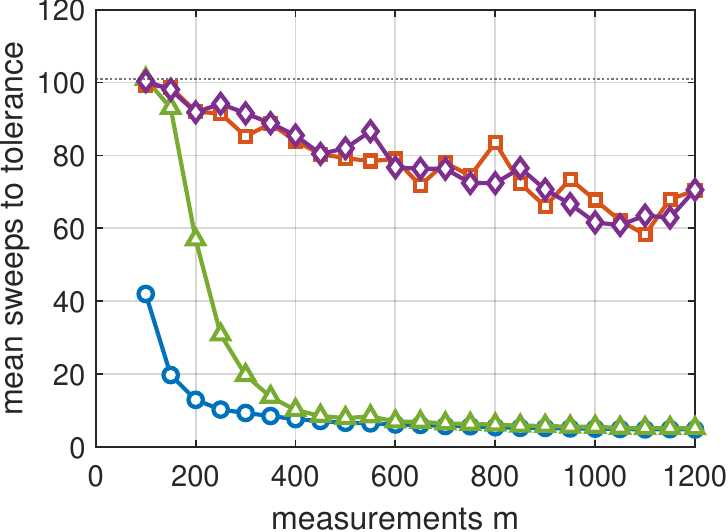}
        \caption{Mean sweeps to tolerance.}
        \label{fig:recovery-measurement-iterations}
    \end{subfigure}

    \caption{Full ALS recovery versus the number of measurements for \(d=12\),
    based on \(100\) trials. Initializers: structure-aware (blue circles),
    spectral (green triangles), random TT (orange squares), and near-zero
    (purple diamonds).}
    \label{fig:recovery-measurement}
\end{figure}

\FloatBarrier

The measurement sweep shows that the projection is most valuable in the
undersampled regime. The structure-aware method succeeds in every trial from
\(m=200\) onward, whereas the spectral initializer becomes comparably reliable
only around \(m=300\)--\(400\); the uninformative baselines remain unreliable
throughout. The isolated spikes in the mean final error of the spectral
initializer are caused by occasional failed trials: successful trials reach
near-machine precision, whereas a failed trial can retain a relative error
close to one and therefore dominate the mean. Once both spectral-type methods
converge, they achieve similar final accuracy with similar sweep counts. Thus,
the principal benefit of the structure-aware initializer is a higher
probability of convergence with fewer measurements, rather than greater
accuracy after convergence.

These conclusions are limited to the noiseless sliding-window model with fixed
supports and oracle TT ranks.

\section{Conclusion}
This work develops a compressed-sensing framework for length-$2^d$ signals
whose binary tensorizations combine bounded TT rank with sparse low-order
interaction structure. For randomly subsampled Walsh--Hadamard measurements,
we established a uniform RIP on this model class and used it to prove
uniqueness of noiseless recovery. We also introduced a structure-aware
initializer for fixed-rank ALS and proved a uniform initialization error bound
with a measurement requirement $\mathcal{O}(d^{s+1}\log d)$ for fixed
interaction order $s$, target accuracy, and failure probability. The experiments
support the predicted qualitative scaling, while showing milder dependence on
$d$, $s$, and $k$ than the worst-case bounds. Most importantly, the
structure-aware initializer yields lower initialization error and reliable ALS
recovery with fewer measurements than the standard spectral initializer.

The analysis suggests two concrete directions for sharper theory. For the
uniform RIP, a single-scale covering argument already gives a measurement bound
of comparable polylogarithmic order to general sparse RIP results. The milder
empirical dependence on $d$, $s$, and $k$ suggests that chaining or a tighter
model-specific covering argument could sharpen the remaining factors. For the
initializer, the current proof controls the projected backprojection over the
full interaction space $\mathcal{V}_s$ and then applies the generic TT-SVD
factor $1+\sqrt{d-1}$. In these experiments, TT-SVD acts as a denoising step.
Analyzing the spectra of the TT unfoldings of
$\mathcal{P}_{\mathcal{V}_s}(\mathcal{A}^*\mathbf{y})$ may lead to a sharper
post-TT-SVD error bound.

The concentration argument also suggests extensions to uniformly subsampled
orthonormal systems with bounded entries. One example is the
one-dimensional discrete Fourier transform applied to the ground-truth vector
$\mathbf{x}^\star$, with its Fourier coefficients sampled uniformly. The
present QTT formulation reshapes a finely discretized one-dimensional signal
of length $2^d$ into a $d$-mode tensor with binary
modes~\cite{khoromskij2011quantics,shinaoka2023multiscale}. Separately, tensors with
binary modes arise intrinsically in quantum many-body models, where low-rank
matrix product representations and few-body interactions suggest a natural,
though not automatic, connection to our joint structural
model~\cite{schollwock2011density,huang2023learning}. This suggests a possible direction for structure-aware recovery in physical inverse problems, but the present
evidence is limited to noiseless, exactly structured sliding-window tensors
with fixed supports and oracle TT ranks. Realizing this potential will require
verifying the simultaneous presence of low TT rank and sparse low-order
interactions, developing physically realizable measurement schemes, and
addressing noise, model mismatch, and unknown ranks.

\section*{Acknowledgments}
This work grew out of the author's master's thesis at the University of Chicago. The author thanks his supervisor, Yuehaw Khoo, for his guidance and many helpful discussions. The author used ChatGPT to assist with language editing and stylistic refinement of the manuscript. All mathematical content was independently reviewed and verified by the author.

\section*{Statements and Declarations}
\noindent\textbf{Funding.}
No funding was received for conducting this study.

\medskip
\noindent\textbf{Competing interests.}
The author has no relevant financial or non-financial interests to disclose.

\medskip
\noindent\textbf{Data availability.}
All data analyzed in this study were generated synthetically. The generated
data supporting the findings are included in the accompanying code repository
at \url{https://github.com/Jingchun-Shao/compressed_sensing_with_QTT}.

\medskip
\noindent\textbf{Code availability.}
The MATLAB code used to generate the data and figures and to perform the
reported recovery experiments is available at
\url{https://github.com/Jingchun-Shao/compressed_sensing_with_QTT}.

\medskip
\noindent\textbf{Author contributions.}
The author developed the methodology and software, performed the analysis and
experiments, and wrote and approved the manuscript.

\bibliographystyle{plain}
\bibliography{ref}

\end{document}